\documentclass{amsart}
\usepackage{amssymb,amsmath,stmaryrd,mathrsfs}
\usepackage{amsthm}
\usepackage[all,2cell]{xy}
\usepackage{tikz}
\usetikzlibrary{cd}
\usepackage{enumitem}
\usepackage{xcolor}
\definecolor{darkgreen}{rgb}{0,0.45,0} 
\usepackage[pagebackref,colorlinks,citecolor=darkgreen,linkcolor=darkgreen]{hyperref}
\usepackage{mathtools}          
\usepackage{microtype}
\usepackage{braket}             
\let\setof\Set
\usepackage{url}                
\usepackage{xspace}             
\usepackage{cleveref,aliascnt}
\usepackage{xr}
\title{Univalence for inverse EI diagrams}
\author{Michael Shulman}
\email{shulman@sandiego.edu}
\address{University of San Diego,
  5998 Alcala Park,
  San Diego, CA 92110,
  USA}
\thanks{This material is based on research sponsored by The United States Air Force Research Laboratory under agreement number FA9550-15-1-0053.  The U.S.~Government is authorized to reproduce and distribute reprints for Governmental purposes notwithstanding any copyright notation thereon.  The views and conclusions contained herein are those of the author and should not be interpreted as necessarily representing the official policies or endorsements, either expressed or implied, of the United States Air Force Research Laboratory, the U.S.~Government, or Carnegie Mellon University.}

\makeatletter

\def\defthm#1#2#3{%
  \newaliascnt{#1}{thm}
  \newtheorem{#1}[#1]{#2}
  \aliascntresetthe{#1}
  \crefname{#1}{#2}{#3}
}
\newtheorem{thm}{Theorem}[section]
\crefname{thm}{Theorem}{Theorems}
\defthm{cor}{Corollary}{Corollaries}
\defthm{lem}{Lemma}{Lemmas}
\theoremstyle{definition}
\defthm{defn}{Definition}{Definitions}
\theoremstyle{remark}
\defthm{rmk}{Remark}{Remarks}
\defthm{eg}{Example}{Examples}

\crefformat{section}{\S#2#1#3}
\Crefformat{section}{Section~#2#1#3}
\crefrangeformat{section}{\S\S#3#1#4--#5#2#6}
\Crefrangeformat{section}{Sections~#3#1#4--#5#2#6}
\crefmultiformat{section}{\S\S#2#1#3}{ and~#2#1#3}{, #2#1#3}{ and~#2#1#3}
\Crefmultiformat{section}{Sections~#2#1#3}{ and~#2#1#3}{, #2#1#3}{ and~#2#1#3}
\crefrangemultiformat{section}{\S\S#3#1#4--#5#2#6}{ and~#3#1#4--#5#2#6}{, #3#1#4--#5#2#6}{ and~#3#1#4--#5#2#6}
\Crefrangemultiformat{section}{Sections~#3#1#4--#5#2#6}{ and~#3#1#4--#5#2#6}{, #3#1#4--#5#2#6}{ and~#3#1#4--#5#2#6}
\crefformat{appendix}{Appendix~#2#1#3}
\Crefformat{appendix}{Appendix~#2#1#3}
\crefrangeformat{appendix}{Appendices~#3#1#4--#5#2#6}
\Crefrangeformat{appendix}{Appendices~#3#1#4--#5#2#6}
\crefmultiformat{appendix}{Appendices~#2#1#3}{ and~#2#1#3}{, #2#1#3}{ and~#2#1#3}
\Crefmultiformat{appendix}{Appendices~#2#1#3}{ and~#2#1#3}{, #2#1#3}{ and~#2#1#3}
\crefrangemultiformat{appendix}{Appendices~#3#1#4--#5#2#6}{ and~#3#1#4--#5#2#6}{, #3#1#4--#5#2#6}{ and~#3#1#4--#5#2#6}
\Crefrangemultiformat{appendix}{Appendices~#3#1#4--#5#2#6}{ and~#3#1#4--#5#2#6}{, #3#1#4--#5#2#6}{ and~#3#1#4--#5#2#6}
\crefname{figure}{Figure}{Figures}

\setitemize[1]{leftmargin=2em}
\setenumerate[1]{leftmargin=*}

\let\c@equation\c@thm
\numberwithin{equation}{section}

\let\ea\expandafter
\def\mdef#1#2{\ea\ea\ea\gdef\ea\ea\noexpand#1\ea{\ea\ensuremath\ea{#2}\xspace}}

\mdef\C{\mathscr{C}}
\let\sC\C
\mdef\D{\mathscr{D}}
\mdef\I{\mathcal{I}}
\mdef\J{\mathcal{J}}
\mdef\K{\mathcal{K}}
\mdef\L{\mathcal{L}}
\mdef\M{\mathcal{M}}
\mdef\N{\mathcal{N}}
\mdef\cZ{\mathcal{Z}}
\mdef\lD{\mathbb{D}}
\mdef\lC{\mathbb{C}}
\mdef\lR{\mathbb{R}}
\mdef\lZ{\mathbb{Z}}
\let\Gm\Gamma
\let\De\Delta
\def\fId{\mathsf{Id}}
\def\nInv{\mathrm{Inv}}
\mdef\cCat{\mathcal{C}\mathit{at}}
\mdef\sSet{\mathrm{sSet}}
\mdef\ssSet{\mathrm{ssSet}}
\mdef\lsSet{\mathrm{s}\mathbb{S}\mathrm{et}}
\mdef\Ibar{\overline{\I}}
\mdef\ttfc{\mathrm{TTFC}_s}
\mdef\cinv{\C\text{-}\nInv}
\mdef\cinvs{\C\text{-}\nInv^S}
\mdef\ssetinv{\sSet\text{-}\nInv}
\mdef\ssetinvs{\sSet\text{-}\nInv^S}
\def\f{_{\mathbf{f}}}
\newcommand{\io}{\ensuremath{(\infty,1)}}
\newcommand{\ig}{\ensuremath{\infty\text{-}\mathrm{Gpd}}\xspace}
\newcommand{\OG}{\ensuremath{\mathcal{O}_G}\xspace}
\newcommand{\blank}{\mathord{\hspace{1pt}\text{--}\hspace{1pt}}}
\newcommand{\strsl}[2]{#1 \mathord{\sslash} #2}
\mdef\ty{\;\mathsf{type}}
\def\r#1{\mathsf{refl}_{#1}}
\mdef\unit{\mathbf{1}}
\mdef\bool{\mathbf{2}}
\mdef\btrue{\mathsf{t}}
\mdef\bfalse{\mathsf{f}}
\let\dn\downarrow
\newcommand{\op}{^{\mathrm{op}}}

\newcommand{\pullback}[1][dr]{\save*!/#1-1.2pc/#1:(-1,1)@^{|-}\restore}
\def\oo{\ensuremath{\infty}}
\newcommand{\too}[1][]{\ensuremath{\overset{#1}{\longrightarrow}}}
\newcommand{\ot}{\ensuremath{\leftarrow}}

\let\into\hookrightarrow
\let\xto\xrightarrow
\def\toiso{\xto{\smash{\raisebox{-.5mm}{$\scriptstyle\sim$}}}}

\makeatletter
\def\prd#1{\@ifnextchar\bgroup{\prd@parens{#1}}{%
    \@ifnextchar\sm{\prd@parens{#1}\@eatsm}{%
    \@ifnextchar\prd{\prd@parens{#1}\@eatprd}{%
    \@ifnextchar\;{\prd@parens{#1}\@eatsemicolonspace}{%
    \@ifnextchar\\{\prd@parens{#1}\@eatlinebreak}{%
    \@ifnextchar\narrowbreak{\prd@parens{#1}\@eatnarrowbreak}{%
      \prd@noparens{#1}}}}}}}}
\def\prd@parens#1{\@ifnextchar\bgroup%
  {\mathchoice{\@dprd{#1}}{\@tprd{#1}}{\@tprd{#1}}{\@tprd{#1}}\prd@parens}%
  {\@ifnextchar\sm%
    {\mathchoice{\@dprd{#1}}{\@tprd{#1}}{\@tprd{#1}}{\@tprd{#1}}\@eatsm}%
    {\mathchoice{\@dprd{#1}}{\@tprd{#1}}{\@tprd{#1}}{\@tprd{#1}}}}}
\def\@eatsm\sm{\sm@parens}
\def\prd@noparens#1{\mathchoice{\@dprd@noparens{#1}}{\@tprd{#1}}{\@tprd{#1}}{\@tprd{#1}}}
\def\lprd#1{\@ifnextchar\bgroup{\@lprd{#1}\lprd}{\@@lprd{#1}}}
\def\@lprd#1{\mathchoice{{\textstyle\prod}}{\prod}{\prod}{\prod}({\textstyle #1})\;}
\def\@@lprd#1{\mathchoice{{\textstyle\prod}}{\prod}{\prod}{\prod}({\textstyle #1}),\ }
\def\tprd#1{\@tprd{#1}\@ifnextchar\bgroup{\tprd}{}}
\def\@tprd#1{\mathchoice{{\textstyle\prod_{(#1)}}}{\prod_{(#1)}}{\prod_{(#1)}}{\prod_{(#1)}}}
\def\dprd#1{\@dprd{#1}\@ifnextchar\bgroup{\dprd}{}}
\def\@dprd#1{\prod_{(#1)}\,}
\def\@dprd@noparens#1{\prod_{#1}\,}

\def\@eatnarrowbreak\narrowbreak{%
  \@ifnextchar\prd{\narrowbreak\@eatprd}{%
    \@ifnextchar\sm{\narrowbreak\@eatsm}{%
      \narrowbreak}}}
\def\@eatlinebreak\\{%
  \@ifnextchar\prd{\\\@eatprd}{%
    \@ifnextchar\sm{\\\@eatsm}{%
      \\}}}
\def\@eatsemicolonspace\;{%
  \@ifnextchar\prd{\;\@eatprd}{%
    \@ifnextchar\sm{\;\@eatsm}{%
      \;}}}

\def\sm#1{\@ifnextchar\bgroup{\sm@parens{#1}}{%
    \@ifnextchar\prd{\sm@parens{#1}\@eatprd}{%
    \@ifnextchar\sm{\sm@parens{#1}\@eatsm}{%
    \@ifnextchar\;{\sm@parens{#1}\@eatsemicolonspace}{%
    \@ifnextchar\\{\sm@parens{#1}\@eatlinebreak}{%
    \@ifnextchar\narrowbreak{\sm@parens{#1}\@eatnarrowbreak}{%
        \sm@noparens{#1}}}}}}}}
\def\sm@parens#1{\@ifnextchar\bgroup%
  {\mathchoice{\@dsm{#1}}{\@tsm{#1}}{\@tsm{#1}}{\@tsm{#1}}\sm@parens}%
  {\@ifnextchar\prd%
    {\mathchoice{\@dsm{#1}}{\@tsm{#1}}{\@tsm{#1}}{\@tsm{#1}}\@eatprd}%
    {\mathchoice{\@dsm{#1}}{\@tsm{#1}}{\@tsm{#1}}{\@tsm{#1}}}}}
\def\@eatprd\prd{\prd@parens}
\def\sm@noparens#1{\mathchoice{\@dsm@noparens{#1}}{\@tsm{#1}}{\@tsm{#1}}{\@tsm{#1}}}
\def\lsm#1{\@ifnextchar\bgroup{\@lsm{#1}\lsm}{\@@lsm{#1}}}
\def\@lsm#1{\mathchoice{{\textstyle\sum}}{\sum}{\sum}{\sum}({\textstyle #1})\;}
\def\@@lsm#1{\mathchoice{{\textstyle\sum}}{\sum}{\sum}{\sum}({\textstyle #1}),\ }
\def\tsm#1{\@tsm{#1}\@ifnextchar\bgroup{\tsm}{}}
\def\@tsm#1{\mathchoice{{\textstyle\sum_{(#1)}}}{\sum_{(#1)}}{\sum_{(#1)}}{\sum_{(#1)}}}
\def\dsm#1{\@dsm{#1}\@ifnextchar\bgroup{\dsm}{}}
\def\@dsm#1{\sum_{(#1)}\,}
\def\@dsm@noparens#1{\sum_{#1}\,}

\makeatother

\begin{document}

\begin{abstract}
  We construct a new model category presenting the homotopy theory of presheaves on ``inverse EI $(\infty,1)$-categories'', which contains universe objects that satisfy Voevodsky's univalence axiom.
  In addition to diagrams on ordinary inverse categories, as considered in previous work of the author, this includes a new model for equivariant algebraic topology with a compact Lie group of equivariance.
  Thus, it offers the potential for applications of homotopy type theory to equivariant homotopy theory.
\end{abstract}


\maketitle

\section{Introduction}
\label{sec:introduction}

\emph{Homotopy type theory}~\cite{hottbook} is a recent subject that synthesizes intensional constructive type theory with homotopy theory.
Among other things, it offers the possibility of using type theory as a ``formal syntax'' for proving homotopy-theoretic theorems, which would apply automatically to any ``homotopy theory'' or \oo-topos~\cite{lurie:higher-topoi,rezk:homotopy-toposes}.
One potential advantage of this over other abstract languages for homotopy theory is that it talks concretely about points and paths, which are then ``compiled'' by an interpretation theorem to diagrammatic arguments.
It also makes available different technical tools, notably \textbf{higher inductive types} (a formal language for cell complexes that avoids small object arguments) and Voevodsky's \textbf{univalence axiom}.

Here we study univalence, which provides a \emph{classifying space for all (small) spaces} (or ``object classifier''~\cite{lurie:higher-topoi}) whose points are \emph{literally} spaces.
Thus, we can work ``representably'' without passing back and forth across equivalences.
For example, defining ``a spectrum'' in type theory automatically defines \emph{the space of spectra}, and thereby also a notion of ``parametrized spectrum'' (a map into the space of spectra).

Together, higher inductive types and univalence enable ``synthetic homotopy theory''; see~\cite{ls:pi1s1,lb:pinsn,lf:emspaces,brunerie:thesis,ffll:blakers-massey} and~\cite[Chapter 8]{hottbook}.
These proofs, written in an intuitive language that involves points and paths, nevertheless ``compile'' automatically into any suitable homotopy theory.
Notably,~\cite{ffll:blakers-massey} was the first purely homotopy-theoretic proof of Blakers--Massey that applies (in principle) to any \oo-topos; afterwards it was translated back into \oo-categorical language~\cite{rezk:hott-blakersmassey}.

However, there is presently a gap in this picture: not all \oo-toposes are known to model univalence in its usual form.\footnote{They do model a less convenient version of it that probably suffices for most applications.}
By~\cite{klv:ssetmodel}, the archetypical \oo-topos of \oo-groupoids does model univalence, and by~\cite{shulman:invdia,shulman:elreedy,cisinski:elegant} so do presheaf \oo-toposes on elegant Reedy categories~\cite{br:reedy}.\footnote{To be precise, all of these models also depend on an ``initiality theorem'', which is known for some type theories~\cite{streicher:semtt} and expected to generalize to all of them.\label{fn:initiality}}
Univalence also passes to slice categories, yielding parametrized homotopy theories; but many important examples are still missing from the list, notably including equivariant homotopy theory.

In this paper I will generalize the univalent models of~\cite{shulman:invdia} to include classical equivariant homotopy theory over a compact Lie group.
Therefore, synthetic homotopy theory applies to equivariant (parametrized) homotopy theory, without modifying the univalence axiom.%
\footnote{See also the parallel line of investigation due to Bordg~\cite{bordg:thesis}.}
By~\cite{elmendorf:theorem}, $G$-equivariant homotopy theory is equivalent to the \oo-topos of diagrams on the orbit category $\OG\op$.
If $G$ is compact Lie, $\OG\op$ is an \emph{inverse EI \io-category}: every endomorphism is an equivalence and the relation ``there is a noninvertible map $y\to x$'' is well-founded.
I will show that type theory with univalence is modeled by the \oo-topos of diagrams on any inverse EI \io-category.

On one hand, this construction is a generalization of~\cite{shulman:invdia} that internalizes in an \io-category.
An ordinary inverse category contains no nonidentity automorphisms, and this remains true for ``internal inverse categories''; but in the latter case there can nevertheless be nontrivial automorphisms ``hidden'' in the \emph{space of objects}.

On the other hand, this construction is also an \emph{iteration} of the ``gluing construction'' (i.e.\ comma categories) from~\cite{shulman:invdia}.
As described in~\cite{shulman:reedy}, inverse diagrams can be obtained by iterated gluing along ``matching object'' functors; here we generalize by gluing along hom-functors of internal categories rather than ordinary ones.

In \cref{sec:preliminaries} we recall basic facts about indexed categories and well-founded recursion.
In \cref{sec:category-theory} we study ``internal inverse categories'' in a general context that can be specialized both to type theory and homotopy theory.
In \crefrange{sec:model}{sec:homotopy-theory} we specialize to homotopy theory, identifying diagrams on such internal categories with previously known models for \oo-toposes of diagrams.
Then in \cref{sec:type-theory} we specialize instead to type theory, proving that our internal diagram categories admit models of homotopy type theory with univalence.
Finally, in \cref{sec:fibrant-categories} we discuss some examples, including equivariant homotopy theory.

There is actually no type theory as such in the main parts of this paper.
We do not even need the statement of univalence, relying instead on the gluing theorem from~\cite{shulman:invdia}.
Type-theoretic syntax will appear only in \cref{sec:fibrant-categories}.
Some familiarity with Quillen model categories and \io-categories is necessary for \crefrange{sec:model}{sec:homotopy-theory}.

I would like to thank Jaap van Oosten for writing~\cite{oosten:functors-wfr} so I could cite it, Geoffroy Horel for several useful conversations about~\cite{horel:model-intsscat}, Pedro Boavida de Brito for sharing an early draft of~\cite{pbb:groth-segal}, and the referee for helpful suggestions on exposition.

\newpage
\section{Preliminaries}
\label{sec:preliminaries}

\subsection{Indexed categories}
\label{sec:indexed-categories}

If \C is any category, a \textbf{\C-indexed category} is a pseudofunctor $\lD:\C\op\to\cCat$, written $X\mapsto \lD^X$ on objects and $f\mapsto f^*$ on morphisms.
A good modern reference is~\cite[Part B]{ptj:elephant1}.
We think of objects of $\lD^X$ as ``$X$-indexed families of objects of \lD'', allowing us to ``do category theory with \lD'' treating \C like the category of sets.
For instance, the following standard definition expresses ``local smallness''.

\begin{defn}\label{defn:locfib}
  Given 
  $A\in \lD^X$ and $B\in \lD^Y$, if the functor
  \begin{align*}
    \C/(X\times Y)\op &\to \mathrm{Set}\\
    ((p,q):Z\to X\times Y) &\mapsto \lD^Z(p^*A,q^*B)
  \end{align*}
  is representable, we denote its representing object by $\lD(A,B) \to X\times Y$. 
\end{defn}

If all such objects and their pullbacks exist, we get associative and unital maps $\lD(A,B) \times_Y \lD(B,C) \to \lD(A,C)$.
In particular, $\lD^1$ is enriched over \C.

\subsection{Well-founded recursion}
\label{sec:wf}

Recall that a relation $\prec$ on a set $I$ is \textbf{well-founded} if the only subset $A\subseteq I$ with the property that $x\in A$ as soon as $y\in A$ for all $y\prec x$ is $I$ itself.
Classically, this is equivalent to the nonexistence of infinite decreasing chains $x_0 \succ x_1 \succ x_2 \succ x_3 \succ \cdots$.

Our well-founded relations will always be \emph{transitive}.
If we define $x\preceq y$ to mean ``$x\prec y$ or $x=y$'', then $\preceq$ is a partial order, which we call a \emph{well-founded poset}.
Since the poset $I$ is a category, it has slice categories such as $I/x$, which is the full sub-poset of $y\in I$ such that $y\preceq x$.
We write $\strsl I x$ for the full sub-poset of $y\in I$ such that $y\prec x$.

If $\prec$ is well-founded and $P(x)$ holds for any $x\in I$ if it holds for all $y\prec x$, then $P(x)$ holds for all $x\in I$.
Similarly, if $F$ assigns to any $x\in I$ and any $g_x:\strsl I x \to Z$ an element of $Z$, there is a unique $g:I\to Z$ with $g(x) = F(x,g|_{\strsl I x})$ for all $x\in I$.
We will also define \emph{functors} by recursion, as in~\cite{oosten:functors-wfr}; the following is an easy generalization.

\begin{thm}\label{thm:wf}
  Let $I$ be a well-founded poset and \cZ be a category with a functor $\Phi:\cZ\to I$.
  Let $F$ be a function which assigns to any $x\in I$ and partial section $G_x:\strsl I x \to \cZ$ of $\Phi$, a cocone under $G_x$ lying $\Phi$-over the canonical cocone under $\strsl I x \into I$ with vertex $x$ (in other words, an extension of $G_x$ to a partial section defined on $I/x$).
  Then there exists a unique section $G:I\to \cZ$ of $\Phi$ such that 
  \begin{enumerate}
  \item For every $x\in I$, $G(x)$ is the vertex of $F(x,G|_{\strsl I x})$, and\label{item:wf1}
  \item For every $y\prec x$, $G(y\prec x)$ is the component of $F(x,G|_{\strsl I x})$ at $y$.\label{item:wf2}
  \end{enumerate}
\end{thm}

\section{Internal inverse categories}
\label{sec:category-theory}

Let \C be a category with the following properties.
\begin{itemize}
\item \C has finite products, including a terminal object $1$.
\item \C has two subcategories whose morphisms we call \textbf{fibrations} and \textbf{prefibrations}.
\item Every isomorphism is a fibration, and every fibration is a prefibration.
\item Every morphism $A\to 1$ is a prefibration.
\item All pullbacks of fibrations and prefibrations exist and are again fibrations or prefibrations, respectively.
\item The dependent product of a prefibration $g$ along a prefibration $f$ exists, is always a prefibration, and is a fibration if $g$ and $f$ are both fibrations.
\end{itemize}
In all cases, the fibrations will be the maps that usually go by that name (in type theory they are sometimes instead called \emph{display maps}).
The prefibrations are an auxiliary class to ensure the existence of pullbacks and dependent products; in \crefrange{sec:model}{sec:homotopy-theory} every map will be a prefibration, while in \cref{sec:type-theory} the prefibrations will coincide with the fibrations.
As usual, $X$ is \textbf{fibrant} if $X\to 1$ is a fibration; by assumption every object is ``prefibrant''.

We write $\lC$ for the \textbf{prefibrant self-indexing}, a \C-indexed category with $\lC^X$ the category of prefibrations with codomain $X$.
Each $\lC^X$ satisfies the above hypotheses.

\begin{lem}\label{thm:dp-pres-fib}
  If we have $X\xto{g} Y \xto{h} Z \xto{k} W$ such that $g$ and $k$ are fibrations and $h$ is a prefibration, then the induced map $k_*(h) \to k_*(h g)$ is a fibration.
\end{lem}
\begin{proof}
  In the language of~\cite{weber:poly-pb}, the following square is a ``distributivity pullback'':
  \[ \xymatrix@-.5pc{k^* k_* Y \ar[r]^-p \ar[d]_q & Y \ar[r]^h & Z \ar[d]^k\\
    k_* Y \ar[rr]_-r && W. } \]
  Thus, by~\cite[Prop.~2.2.3]{weber:poly-pb}, the mate $r_! q_* p^* \to k_*  h_!$ is an isomorphism.
  Now our map $k_*(h) \to k_*(h g)$ is the composite $k_* h_!(g) \cong r_! q_* p^*(g) \to r_!(1) = r$.
  Since $g$ is a fibration, so is its pullback $p^*(g)$, and since $q$ is a fibration (being a pullback of the fibration $k$), so is $q_* p^*(g)$.
  Finally, $r_!$ doesn't change the underlying map in \C.
\end{proof}

If $I$ is a well-founded poset and $A\in \C^{I\op}$, its \textbf{matching object} at $x\in I$ is the limit of its restriction to $\strsl x {I\op}$ (or equivalently $(\strsl I x)\op$), if it exists:
\[ M_x A = \lim_{\strsl x {I\op}} A. \]
We say $A$ is \textbf{Reedy fibrant} if $M_x A$ exists and the induced map $A_x\to M_x A$ is a fibration for all $x$.
More generally, $A\to B$ is a \textbf{Reedy fibration} if each $M_x A$, $M_x B$, and the pullback $M_x A \times_{M_x B} B_x$ exist, and each induced map $A_x \to M_x A \times_{M_x B} B_x$ is a fibration.
Similarly, we have \textbf{Reedy prefibrations} and \textbf{Reedy prefibrant} objects.
The following are simplified versions of~\cite[Defs.~11.4 and 11.9 and Lem.~11.8]{shulman:invdia}.

\begin{defn}\label{defn:prereedy-limits}
  For a well-founded poset $I$, we say \C has \textbf{pre-Reedy $I\op$-limits} if
  \begin{enumerate}
  \item Any Reedy prefibrant $A\in\C^{I\op}$ has a limit, and 
  \item If $A,B\in\C^{I\op}$ are Reedy prefibrant and $f:A\to B$ is a Reedy fibration, $\lim f : \lim A \to \lim B$ is a fibration.\label{item:rlim2}
In particular, if $A$ is Reedy fibrant, $\lim A$ is fibrant.
  \end{enumerate}
\end{defn}

\begin{defn}\label{defn:pre-admissible}
  A well-founded poset $I$ is \textbf{pre-admissible} for \C if \C has pre-Reedy $(\strsl I x)\op$-limits for all $x\in I$.
\end{defn}

\begin{lem}
  If $I$ is finite, then any \C satisfying our hypotheses above has pre-Reedy $I\op$-limits.
  Thus, if each $\strsl I x$ is finite, then $I$ is pre-admisible for any \C.
\end{lem}

\begin{defn}\label{defn:cinversecat}
  A \textbf{\C-inverse category} \I consists of the following.
  \begin{enumerate}
  \item A set $\I_0$ of ``objects'' equipped with a transitive well-founded relation $\prec$.\label{item:cic1}
  \item For each $x\in \I_0$, a specified object $\I(x)\in \C$.
  \item For each $x,y\in\I_0$ with $y\prec x$, a span $\I(x) \ot \I(x,y) \to \I(y)$, in which $\I(x,y) \to \I(x)$ is a prefibration.
    (In particular, there is no $\I(x,x)$.)
  \item For each $x,y,z\in\I_0$ with $z\prec y\prec x$, a map $\I(x,y) \times_{\I(y)} \I(y,z) \to \I(x,z)$ over $\I(x)$ and $\I(z)$ (the pullback existing because $\I(y,z) \to \I(y)$ is a prefibration).
  \item For each $x,y,z,w\in\I_0$, the evident associativity square commutes.
  \end{enumerate}
\end{defn}


\begin{eg}\label{eg:wf-cinv}
  If $\I(x) = \I(x,y) = 1$ for all $x,y$, the only datum is $(\I_0,\prec)$.
\end{eg}

\begin{eg}\label{eg:inv-cinv}
  If \C has pullback-stable coproducts, an ordinary inverse category $I$ yields a \C-inverse category with the same objects, $\I(x)=1$ for all $x$, and $\I(x,y) = \coprod_{I(x,y)} 1$.
\end{eg}



\begin{rmk}
  The referee has pointed out that when \C is infinitary-extensive~\cite{clw:ext-dist}, a \C-inverse category is equivalently an ordinary \C-internal category \K together with an identity-reflecting functor $\K\to\Delta(\I_0)$, plus a prefibration condition.
  Here $\Delta(\I_0)$ is the discrete internal category on the ordinary category $\I_0$, and ``identity-reflecting'' means that the square
witnessing the preservation of identities is a pullback.
\end{rmk}

\begin{defn}\label{defn:idiag}
  Let \I be a \C-inverse category.
  The \C-indexed category $\lC^\I$ of \textbf{\I-diagrams} is defined as follows.
  \begin{itemize}
  \item An object $A\in (\lC^\I)^X$ consists of
    \begin{enumerate}
    \item For each $x\in \I_0$, a span $X \ot A_x \to \I(x)$ in which $A_x \to X$ is a prefibration.
    \item For each $y\prec x$ in $\I_0$, a map $A_x \times_{\I(x)} \I(x,y) \to A_y$ over $X$ and $\I(y)$.
    \item For each $z\prec y\prec x$, the evident associativity square commutes.
    \end{enumerate}
  \item A map in $(\lC^\I)^X$ consists of span maps $A_x \to B_x$ commuting with the actions.
  \item Reindexing along $f:Y\to X$ in \C is given by pullback of $A_x \to X$.
  \end{itemize}
  For brevity, we will write $\C^\I$ in place of $(\lC^\I)^1$.
\end{defn}

\begin{eg}
  For \I as in \cref{eg:wf-cinv,eg:inv-cinv}, $(\lC^\I)^X$ reduces to the usual category of diagrams in the category of prefibrations over $X$.
\end{eg}

\begin{defn}
  For any \C-inverse category \I and any subset $\J_0 \subseteq \I_0$, there is a \textbf{full \C-inverse subcategory} \J of \I defined by $\J(x) = \I(x)$, $\J(x,y) = \I(x,y)$, and so on.
  We say \J is \textbf{down-closed} if $\J_0$ is downwards closed under $\prec$.

  Given $x\in\I_0$, we denote by $x/\I$ and $\strsl x\I$ the down-closed full \C-inverse subcategories of \I determined by $\I_0/x = \setof{y|y\preceq x}$ and $\strsl{\I_0}{x} = \setof{y|y\prec x}$, respectively.
\end{defn}

\begin{eg}\label{thm:collage}
  For any \I and $x\in\I_0$, the spans $\I(x) \ot \I(x,y) \to \I(y)$ and corresponding actions of $\I(y,z)$ assemble \emph{precisely} into an object of $(\lC^{\strsl x \I})^{\I(x)}$, which we denote $\I(x,-)$.
  Said differently, the additional data required to extend a \C-inverse category \J by adding a new object $x$ ``at the top'' consists precisely of an object $\I(x)\in\C$ and an object of $(\lC^{\J})^{\I(x)}$.
  Categorically speaking, $\I$ is the \emph{collage} of $\I(x,\blank)\in (\lC^{\J})^{\I(x)}$, regarded as a sort of ``profunctor'' from $\J$ to the $\I(x)$-indexed terminal category.
\end{eg}

For any $x\in\I_0$ there is a \C-indexed forgetful functor $\lC^{x/\I}\to \lC^{\strsl x \I}$.
By definition, to extend $A\in (\lC^{\strsl x \I})^X$ to an object of $(\lC^{x /\I})^X$ we must give:
\begin{enumerate}
\item An object $A_x$ and a span $X \ot A_x \to \I(x)$ in which $A_x \to X$ is a prefibration.
\item For each $y\prec x$, a map $A_x \times_{\I(x)} \I(x,y) \to A_y$ over $X$ and $\I(y)$, where $A_y$ is given as part of the given diagram $A\in (\lC^{\strsl x \I})^X$.
\end{enumerate}
such that
\begin{enumerate}[resume]
\item The evident associativity squares commute.
\end{enumerate}
Now a map $A_x \times_{\I(x)} \I(x,y) \to A_y$ over $X$ and $\I(y)$ is equivalently a map $A_x \times_{\I(x)} \I(x,y) \to A_x \times_X A_y$ over $A_x$ and $\I(y)$.
The associativity diagrams then say that these maps assemble into a morphism in $(\lC^{\strsl x \I})^{A_x}$ from the reindexing of $\I(x,-)$ along $A_x \to \I(x)$ to the reindexing of $A$ along $A_x \to X$.
\cref{defn:locfib} then gives:

\begin{thm}\label{thm:reedy-char}
  Given \I, an $x\in\I_0$, and $A\in (\lC^{\strsl x \I})^X$, if the indexed hom-object $\lC^{\strsl x \I}(\I(x,-),A)$ exists, then to extend $A$ to an object of $(\lC^{x/\I})^X$ we must give
  \begin{enumerate}
  \item An object $A_x$ and
  \item A map $A_x \to \lC^{\strsl x \I}(\I(x,-),A)$ such that
  \item The composite $A_x \to \lC^{\strsl x \I}(\I(x,-),A) \to \I(x) \times X \to X$ is a prefibration.
  \end{enumerate}
  Similarly, given $A,B\in (\lC^{x/ \I})^X$ and a map $\bar f$ between the restrictions $\bar{A}$ and $\bar{B}$ of $A$ and $B$ to $\strsl x\I$, if $\lC^{\strsl x \I}(\I(x,-),\bar A)$ and $\lC^{\strsl x \I}(\I(x,-),\bar B)$ exist, to extend $\bar f$ to a map $f:A\to B$ we need
  \begin{enumerate}
  \item A map $f_x :A_x \to B_x$ such that
  \item The following square commutes:
    \begin{equation*}
      \vcenter{\xymatrix@C=5pc@R=1.5pc{
          A_x \ar[r]^{f_x}\ar[d] &
          B_x \ar[d]\\
          \lC^{\strsl x \I}(\I(x,-),\bar A)\ar[r]_{\lC^{\strsl x \I}(\I(x,-),\bar f)} &
          \lC^{\strsl x \I}(\I(x,-),\bar B).
        }}
    \end{equation*}
  \end{enumerate}
\end{thm}

\begin{defn}
  Given $A \in (\lC^{\strsl x\I})^X$, if $\lC^{\strsl x \I}(\I(x,-),A)$ exists, we call it the \textbf{matching object} of $A$ at $x$ and denote it by $M_x A$.
  If $A \in(\lC^{\I})^X$ instead, we write $M_x A$ for the matching object of its restriction to $\strsl x \I$.
\end{defn}

\begin{defn}\label{defn:reedy-fibration}
  An $A\in (\lC^\I)^X$ is \textbf{Reedy fibrant} (resp.~\textbf{Reedy prefibrant}) if each $M_x A$ exists and each map $A_x \to M_x A$ is a fibration (resp.~a prefibration).
  More generally, $f:A\to B$ in $(\lC^\I)^X$ is a \textbf{Reedy fibration} if each $M_x A$ and $M_x B$ and each pullback $B_x \times_{M_x B} M_x A$ exist, and $A_x \to B_x \times_{M_x B} M_x A$ is a fibration.
\end{defn}

The following definition may look curious, but it is essential for \cref{thm:homs}.
A reader who wants to understand it better immediately may skip forward to \cref{sec:fibrant-categories}.

\begin{defn}\label{defn:invcat-fibrant}
  A \C-inverse category \I is \textbf{fibrant} if each $\I(x)$ is fibrant and each $\I(x,-)\in (\lC^{\strsl x \I})^{\I(x)}$ is Reedy fibrant.
\end{defn}

If $\J\subseteq \I$ is a down-closed full \C-inverse subcategory, then $\strsl x \J = \strsl x \I$ for any $x\in\J_0$, so restriction $\lC^\I\to\lC^\J$ preserves matching objects.
Thus, any down-closed full \C-inverse subcategory of a fibrant \I is again fibrant; this includes $x/\I$ and $\strsl x \I$.


\begin{lem}\label{thm:hom-limit}
  For any \I and $A\in (\lC^\I)^X$ and $B\in (\lC^\I)^Y$, if $\lC^{x/\I}(A,B)$ exists for all $x\in\I_0$, then in $\C/(X\times Y)$ we have
  \[ \lC^\I(A,B) \cong \lim_{x\in\I_0} \lC^{x/\I}(A,B) \]
  in the strong sense that if either exists, so does the other and they are isomorphic.
\end{lem}
\begin{proof}
  A morphism between \I-diagrams is determined by compatible morphisms between their restrictions to each $x/\I$, so both sides represent the same functor.
\end{proof}


\begin{lem}\label{thm:cinv-wf}
  For \C-inverse categories \I and \J, define $\J\prec \I$ to mean that $\J = \strsl x\I$ for some $x\in \I$.
  Then the relation $\prec$ is transitive and well-founded.
\end{lem}
\begin{proof}
  Transitivity is because $\strsl y {(\strsl x \I)} = \strsl y \I$.
  Well-foundedness is because the class of sets with well-founded relations is itself well-founded with an analogous $\prec$.
\end{proof}

\begin{thm}\label{thm:homs}
  Suppose \I is fibrant, $A\in (\lC^\I)^X$ and $B\in (\lC^\I)^Y$ are Reedy prefibrant, and $\lC^{X\times Y}$ has pre-Reedy $\I_0\op$-limits.
  Then the hom-object $\lC^\I(A,B)$ exists, and $\lC^\I(A,B)\to X\times Y$ is a prefibration, which is a fibration if $A$ and $B$ are Reedy fibrant.
\end{thm}
\begin{proof}
  By well-founded induction on the relation $\prec$ from \cref{thm:cinv-wf}, when proving the claim for \I we may assume it for each $\strsl x \I$.
  We begin by showing that it is also true for each $x/\I$.
  Thus suppose given Reedy prefibrant $A\in (\lC^{x/\I})^X$ and $B\in (\lC^{x/\I})^Y$.
  By the inductive hypothesis, we have prefibrations
  \begin{align*}
    M_x A = \lC^{\strsl x \I}(\I(x,-),A) &\to \I(x)\times X\\
    M_x B = \lC^{\strsl x \I}(\I(x,-),B) &\to \I(x)\times Y\\
    \lC^{\strsl x \I}(A,B) &\to X\times Y
  \end{align*}
  Since $\I(x)$ is fibrant, $Y$ is (like every object) prefibrant, and $A$ and $B$ are Reedy prefibrant, we have prefibrations $M_x A \to X$ and $M_x B \to Y$ and $\lC^{\strsl x \I}(A,B) \to X$ and $A\to M_x A$ and $B\to M_x B$.
  If $A$ and $B$ are Reedy fibrant, all of these are fibrations.

  Now by the definition of $M_x$ as an indexed hom-object, the observation after \cref{defn:locfib} about composition for the latter gives us a composition map
  \[ c : M_x A \times_X \lC^{\strsl x \I}(A,B) \to M_x B, \]
  the pullback existing because $M_x A \to X$ is a prefibration.
  Both projections
  \begin{align*}
    \pi_1 : M_x A \times_X \lC^{\strsl x \I}(A,B) &\to M_x A\\
    \pi_2 : M_x A \times_X \lC^{\strsl x \I}(A,B) &\to \lC^{\strsl x \I}(A,B).
  \end{align*}
  are prefibrations, since they are pullbacks of the prefibrations $\lC^{\strsl x \I}(A,B) \to X$ and $M_x A \to X$ respectively; and if $A$ is Reedy fibrant, then $\pi_2$ is a fibration.

  Let $\pi_1^* A_x$ and $c^* B_x$ denote the pullbacks of $A_x$ and $B_x$ along $\pi_1$ and $c$ respectively, as in \cref{fig:homs}.
  Then we have induced maps
  \begin{align*}
    \pi_1^* A_x &\to M_x A \times_X \lC^{\strsl x \I}(A,B)\\
    c^* B_x &\to M_x A \times_X \lC^{\strsl x \I}(A,B).
  \end{align*}
  By assumption on \C, their local exponential $(c^*B_x)^{\pi_1^*A_x}$ exists and is a prefibration over $M_x A \times_X \lC^{\strsl x \I}(A,B)$.
  And since $\pi_2$ is a prefibration, the dependent product $(\pi_2)_*{\left((c^*B_x)^{\pi_1^*A_x}\right)}$ exists and is a prefibration.
  All of these maps are also fibrations if $A$ and $B$ are Reedy fibrant.
  Thus, the composite prefibration
  \begin{equation}
    (\pi_2)_* {\left((c^*B_x)^{\pi_1^*A_x}\right)} \to \lC^{\strsl x \I}(A,B) \to X\times Y\label{eq:homs}
  \end{equation}
  is a fibration if $A$ and $B$ are Reedy fibrant;
  I claim it has the desired universal property.

  By the universal property of $(\pi_2)_*$, for any $Z$, to give a map $Z\to (\pi_2)_* {\left((c^*B_x)^{\pi_1^*A_x}\right)}$ is equivalent to giving a map $Z\to \lC^{\strsl x \I}(A,B)$ along with a map $M_x A \times_X Z \to (c^*B_x)^{\pi_1^*A_x}$ over $M_x A \times_X \lC^{\strsl x \I}(A,B)$.
  And by the universal property of $(c^*B_x)^{\pi_1^*A_x}$, to give the latter is equivalent to giving a map $A_x \times_X Z \to c^* B_x$ over $M_x A \times_X \lC^{\strsl x \I}(A,B)$, or equivalently a map $A_x \times_X Z \to B_x$ over $c$.
  Applying the universal property of pullbacks again in reverse, this is equivalent to giving a map $A_x\times_X Z \to B_x\times_Y Z$ over the induced map $M_x A \times_X \lC^{\strsl x \I}(A,B) \to M_x B \times_Y \lC^{\strsl x \I}(A,B)$.

  Now, by the universal property of $\lC^{\strsl x \I}(A,B)$, a map $Z\to \lC^{\strsl x \I}(A,B)$ is equivalent to a map $(p,q):Z\to X\times Y$ together with a map $p^* \bar{A}\to q^*\bar{B}$ over $Z$ between the pullbacks of the restrictions of $A$ and $B$ to $\strsl x\I$ (which are also the restrictions of the pullbacks).
  In particular, there is a universal map $\bar A \times_X \lC^{\strsl x \I}(A,B) \to \bar B \times_Y \lC^{\strsl x \I}(A,B)$ over $\lC^{\strsl x \I}(A,B)$, from which the above-mentioned map $M_x A \times_X \lC^{\strsl x \I}(A,B) \to M_x B \times_Y \lC^{\strsl x \I}(A,B)$ is obtained by the functor $M_x$.
  Thus, to lift the latter map to a map $A_x\times_X Z \to B_x\times_Y Z$, i.e.\ a map $p^* A_x \to q^* B_x$, is the same as to lift its pullback $M_x A \times_X Z \to M_x B \times_Y Z$.
  Finally, by \cref{thm:reedy-char} this is equivalent to lifting the map $p^* \bar{A}\to q^*\bar{B}$ to a map $p^* A \to q^* B$, as desired.

  This concludes the proof for $x/\I$.
  By \cref{thm:hom-limit} to show that $\lC^\I(A,B)$ exists, we may show that $\lim_{x\in\I_0} \lC^{x/\I}(A,B)$ exists.
  Since $x\mapsto \lC^{x/\I}(A,B)$ is a diagram in $\lC^{X\times Y}$ indexed by $\I_0\op$, and $\lC^{X\times Y}$ has pre-Reedy $\I_0\op$-limits, it suffices to show that this diagram is Reedy prefibrant, and Reedy fibrant if $A$ and $B$ are Reedy fibrant.
  So we must show that $\lC^{x/\I}(A,B) \to \lim_{y\prec x} \lC^{y/\I}(A,B)$ is a prefibration which is a fibration if $A$ and $B$ are Reedy fibrant.
  But by \cref{thm:hom-limit}, we have $\lim_{y\prec x} \lC^{y/\I}(A,B) \cong \lC^{\strsl x\I}(A,B)$, and the above construction of $\lC^{x/\I}(A,B)$ showed that $\lC^{x/\I}(A,B) \to \lC^{\strsl x\I}(A,B)$ was a prefibration, and a fibration if $A$ and $B$ are Reedy fibrant.
  The claim follows from the fact that $\lC^{X\times Y}$ has pre-Reedy $\I_0\op$-limits.
\end{proof}

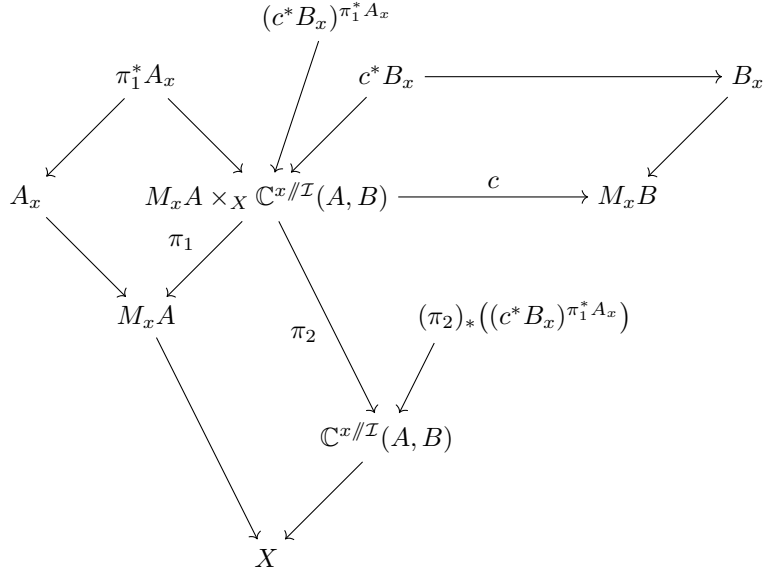
\begin{figure}
  \centering
  \begin{tikzpicture}[->,scale=.8]
    \node (HAxAB) at (0,0) {$M_x A \times_X \lC^{\strsl x \I}(A,B)$};
    \node (HB) at (6,0) {$M_x B$};
    \draw (HAxAB) -- node[auto] {$c$} (HB);
    \node (HA) at (-2,-2) {$M_x A$};
    \draw (HAxAB) -- node[auto,swap] {$\pi_1$} (HA);
    \node (AB) at (2,-4) {$\lC^{\strsl x \I}(A,B)$};
    \draw (HAxAB) -- node[auto,swap] {$\pi_2$} (AB);
    \node (X) at (0,-6) {$X$};
    \draw (AB) -- (X);
    \draw (HA) -- (X);
    \node (B') at (8,2) {$B_x$};
    \draw (B') -- (HB);
    \node (A') at (-4,0) {$A_x$};
    \draw (A') -- (HA);
    \node (cB') at (2,2) {$c^* B_x$};
    \draw (cB') -- (HAxAB);
    \draw (cB') -- (B');
    \node (piA') at (-2,2) {$\pi_1^*A_x$};
    \draw (piA') -- (HAxAB);
    \draw (piA') -- (A');
    \node (Q) at (1,3) {$(c^*B_x)^{\pi_1^*A_x}$};
    \draw (Q) -- (HAxAB);
    \node (piQ) at (3,-2) {$(\pi_2)_*\mathrlap{\left((c^*B_x)^{\pi_1^*A_x}\right)}$};
    \draw (piQ) -- (AB);
  \end{tikzpicture}
  \caption{The construction of hom-objects}
  \label{fig:homs}
\end{figure}

\begin{cor}\label{thm:matching}
  If \I is fibrant and $\I_0\op$ is pre-admissible for \C, and $A\in (\lC^{\strsl x \I})^X$ is Reedy prefibrant, then the matching object $M_x A$ exists, and the map $M_x A \to \I(x) \times X$ is a prefibration which is a fibration if $A$ is Reedy fibrant.\qed
\end{cor}

Thus, under the hypotheses of \cref{thm:matching}, the assumption in \cref{defn:reedy-fibration} that the matching objects exist is unnecessary for Reedy prefibrant objects:
If $A$ is Reedy (pre)fibrant below some stage $x$, then $M_x A$ automatically exists.
Also, the third condition in \cref{thm:reedy-char} is unneeded for defining Reedy prefibrant objects, since each map in the composite shown is a fibration or a prefibration.
We record this:

\begin{cor}\label{thm:matching-gluing}
  If \I is fibrant and $\I_0\op$ is pre-admissible for \C, $x\in\I_0$, and $A\in (\lC^{\strsl x \I})^X$ is Reedy prefibrant, to extend $A$ to a Reedy prefibrant object of $(\lC^{x/\I})^X$ we must give
  (i) an object $A_x$ and
  (ii) a prefibration $A_x \to M_x A$. 
  \qed
\end{cor}


\begin{thm}\label{thm:homs-fibration}
  Suppose \I is fibrant, $A\in (\lC^\I)^X$ is Reedy fibrant, $B,B'\in (\lC^\I)^Y$ are Reedy prefibrant, and $\lC^{X\times Y}$ has pre-Reedy $\I_0\op$-limits.
  If $g:B \to B'$ is a Reedy fibration, then the induced map $\lC^\I(A,B) \to \lC^\I(A,B')$ is a fibration.
\end{thm}
\begin{proof}
  As in \cref{thm:homs}, we assume the statement for all $\strsl x \I$ and prove it for \I.

  In \cref{thm:hom-limit} we constructed $\lC^\I(A,B)$ as a limit of $\lC^{x/\I}(A,B)$ over $\I_0\op$.
  Since \C has pre-Reedy $\I_0\op$-limits, to show the statement it suffices to show that the induced map of $\I_0\op$-diagrams is a Reedy fibration.
  In \cref{thm:homs} we identified the ordinary matching objects of $x\mapsto \lC^{x/\I}(A,B)$ with $\lC^{\strsl x \I}(A,B)$, so we must show that
  \begin{equation}
    \lC^{x/\I}(A,B) \too \lC^{x/\I}(A,B') \times_{\lC^{\strsl x \I}(A,B')} \lC^{\strsl x \I}(A,B)\label{eq:homsfib-map}
  \end{equation}
  is a fibration.
  For this purpose we construct~\eqref{eq:homsfib-map} as follows.

  Let $c'$, $\pi_1'$, and $\pi_2'$ denote the morphisms for $B'$ analogous to $c$, $\pi_1$, and $\pi_2$ for $B$.
  We start with the map $B_x \to M_x B \times_{M_x B'} B_x' = (M_x g)^* B_x'$, which is a fibration since $g$ is a Reedy fibration, and lies in the slice over $M_x B$.
  Applying $c^*$, which preserves fibrations, and noting that $M_x g\circ c= c' \circ \left(1\times \lC^{\strsl x \I}(A,g)\right)$, we get a fibration
  \begin{equation}
    c^* B_x \to c^* (M_x g)^* B_x'
    \cong \left(1\times \lC^{\strsl x \I}(A,g)\right)^* (c')^* B_x'.\label{eq:mpf3}
  \end{equation}
  over $M_x A \times_X \lC^{\strsl x \I}(A,B)$.
  Since $A$ is Reedy fibrant, $\pi_1^*A_x \to M_x A \times_X \lC^{\strsl x \I}(A,B)$ is a fibration, so the local exponential by it preserves fibrations.
  Applying this local exponential to~\eqref{eq:mpf3}, we obtain a fibration
  \begin{align*}
    (c^* B_x)^{\pi_1^*A_x}
    &\to \left(\left(1\times \lC^{\strsl x \I}(A,g)\right)^* (c')^* B_x'\right)^{\pi_1^*A_x}\\
    &\cong \left(1\times \lC^{\strsl x \I}(A,g)\right)^* {\left(((c')^*B'_x)^{{\pi'_1}^*A_x}\right)}
  \end{align*}
  where the isomorphism is because $\pi_1' \circ \left(1\times \lC^{\strsl x \I}(A,g)\right) = \pi_1$ and pullback preserves local exponentials.
  Now we can obtain~\eqref{eq:homsfib-map} as the composite
  \begin{align*}
    \lC^{x/\I}(A,B)
    &= (\pi_2)_* (c^* B_x)^{\pi_1^*A_x}\\
    &\to (\pi_2)_* \left(1\times \lC^{\strsl x \I}(A,g)\right)^* {\left(((c')^*B'_x)^{{\pi'_1}^*A_x}\right)}\\
    &\cong \left(\lC^{\strsl x \I}(A,g)\right)^* (\pi_2')_* {\left(((c')^*B'_x)^{{\pi'_1}^*A_x}\right)}\\
    &= \left(\lC^{\strsl x \I}(A,g)\right)^* \left(\lC^{x/\I}(A,B')\right)\\
    &= \lC^{x/\I}(A,B') \times_{\lC^{\strsl x \I}(A,B')} \lC^{\strsl x \I}(A,B).
  \end{align*}
  Here the isomorphism is the Beck-Chevalley isomorphism for the pullback square
  \begin{equation*}\label{eq:mpf1.5}
  \vcenter{\xymatrix{
      M_x A \times_X \lC^{\strsl x \I}(A,B) \ar[r]^-{\pi_2} \ar[d]_{1\times \lC^{\strsl x \I}(A,g)} &
      \lC^{\strsl x \I}(A,B)\ar[d]^{\lC^{\strsl x \I}(A,g)}\\
      M_x A \times_X \lC^{\strsl x \I}(A,B')\ar[r]_-{\pi_2'} &
      \lC^{\strsl x \I}(A,B').
      }}
  \end{equation*}
  Since $A$ is Reedy fibrant, $\pi_2$ is a fibration, so $(\pi_2)_*$ preserves fibrations; thus the above composite is a fibration, as desired.
\end{proof}

\begin{cor}\label{thm:matching-fibration}
  If \I is fibrant and $\I_0\op$ is pre-admissible for \C, and $A\to B$ is a Reedy fibration in $(\lC^{\strsl x \I})^X$, then the induced map $M_x A \to M_x B$ is a fibration.\qed
\end{cor}

\begin{cor}\label{thm:fib-objwise}
  If \I is fibrant and $\I_0\op$ is pre-admissible for \C, and $A\to B$ is a Reedy fibration, then each $A_x \to B_x$ is also a fibration.\qed
\end{cor}

\section{Model categories of inverse diagrams}
\label{sec:model}

Now let \C be a \emph{type-theoretic model category}; as in~\cite{shulman:invdia} this means a right proper model category in which limits preserve cofibrations and pullback along any fibration has a right adjoint.
We apply the theory of \cref{sec:category-theory} with the model-categorical fibrations as the fibrations and all morphisms as the prefibrations.
We observe:

\begin{lem}
  A type-theoretic model category has pre-Reedy $I\op$-limits for any $I$.
  Hence any $I$ is pre-admissible for \C.
\end{lem}
\begin{proof}
  As in~\cite[Lemma 11.5]{shulman:invdia}, the limit functor is right Quillen.
\end{proof}

Our goal is to prove the following.

\begin{thm}\label{thm:model-structure}
  If \C is a type-theoretic model category and \I is a fibrant \C-inverse category, then $\C^\I$ is a model category in which:
  \begin{itemize}
  \item As in \cref{defn:reedy-fibration}, $A\to B$ is a fibration or acyclic fibration if each map $A_x \to B_x \times_{M_x B} M_x A$ is so.
  \item The cofibrations, weak equivalences, and acyclic cofibrations are levelwise.
  \end{itemize}
\end{thm}

The proof will be by well-founded induction, using \cref{thm:reedy-char} and taking limits.
However, we can only construct limits of model structures that are sufficiently ``algebraic''.
Say that a model category is \textbf{cloven} if we have chosen particular factorizations and liftings; a \textbf{strict functor} between cloven model categories is one that preserves all three classes of maps and the chosen factorizations and lifts.

\begin{lem}[{\cite[{Theorem \ref{reedy:thm:clovenlim}}]{shulman:reedy}}]\label{thm:model-limits}
  The category of cloven model categories and strict functors has limits, which are created by the forgetful functor to \cCat.\qed
\end{lem}

\begin{lem}[{\cite[{Theorem \ref{reedy:thm:model}}]{shulman:reedy}}]\label{thm:model-gluing}
  If \M and \N are model categories and $F:\M\to\N$ preserves limits and acyclic fibrations, then the glued category $(\N\dn F)$ has a model structure in which
  \begin{itemize}
  \item A map from $N\to F M$ to $N' \to F M'$ is a weak equivalence, cofibration, or acyclic cofibration just when $M\to M'$ and $N\to N'$ are both such.
  \item A map from $N\to F M$ to $N' \to F M'$ is a fibration or acyclic fibration just when the induced map $N \to F M \times_{F M'} N'$ is a fibration or acyclic fibration, respectively.
  \end{itemize}
  If \M and \N are cloven, so is $(\N\dn F)$, and $(\N\dn F) \to \M$ is strict.
\end{lem}
\begin{proof}[Sketch of proof]
  Limits and colimits in $(\N\dn F)$ are easy, and the two weak factorization systems are defined in the usual Reedy manner.
  The assumption on $F$ implies that a map in $(\N\dn F)$ is a weak equivalence and a fibration just when $N \to F M \times_{F M'} N'$ is an acyclic fibration, which ensures that the weak factorization systems fit together.
\end{proof}

\begin{lem}\label{thm:matching-model-gluing}
  Let \I be a fibrant \C-inverse category, $x\in \I$, and suppose $\C^{\strsl x \I}$ is a model category with the classes of maps from \cref{thm:model-structure}.
  Then $M_x:\C^{\strsl x \I} \to \C$ preserves limits and acyclic fibrations, and $(\C\dn M_x)$ is equivalent to $\C^{x/\I}$.
\end{lem}
\begin{proof}
  Any hom-functor $\lD(A,-)$ preserves all limits that the reindexing functors of \lD do, so $M_x$ preserves all limits.
  It preserves acyclic fibrations by the same argument as in \cref{thm:homs-fibration}, since all the ingredients therein also preserve acyclic fibrations.
  The final statement follows from \cref{thm:reedy-char}.
\end{proof}

\begin{proof}[Proof of \cref{thm:model-structure}]
  Choose factorizations and liftings to make \C cloven.
  We argue by well-founded induction as in \cref{thm:homs,thm:homs-fibration}, but there are some subtleties.
  Firstly, since we need to carry along the cloven structures, we are not just proving a statement but constructing a function.
  Secondly, for the inductive step we will need to know not only that each $\C^{\strsl x\I}$ is a model category, but that these model structures ``fit together'' as $x$ varies; so we actually must construct a \emph{functor} on the well-founded \emph{poset} of \C-inverse categories.
  Thirdly, finding a codomain for this functor is a bit tricky.
  We might guess the category of cloven model categories and strict functors, so that our functor would send \I to $\C^\I$ and the relation $(\strsl x \I) \prec \I$ to a strict restriction functor $\C^\I \to \C^{\strsl x\I}$.
  But this doesn't seem to work, because to define $\C^\I$ as a limit of $\C^{x/\I}$ we need to know that the restriction functors $\C^{x/\I} \to \C^{y/\I}$ are also strict, which requires that the inductive hypothesis ``know'' something about $M_y$.

  Thus, we will actually define a \emph{dependently typed functor} as in \cref{thm:wf}, i.e.\ a section of some given functor $\Phi:\cZ \to \cinv$, where \cinv is the well-founded poset of \C-inverse categories.
  We let an object of \cZ over $\I\in\cinv$ be a cloven model structure on $\C^\I$ with the given fibrations, cofibrations, and weak equivalences.
  Note that by \cref{thm:model-gluing,thm:matching-model-gluing}, if we have such a model structure on $\C^{\strsl x \I}$, we can glue it along $M_x$ to get such a model structure on $\C^{x/\I}\cong (\C\dn M_x)$ such that the projection $\C^{x/\I} \cong (\C\dn M_x) \to \C^{\strsl x \I}$ is strict.
  We define a morphism of \cZ over $(\strsl x\I)\prec \I$ to be the assertion that $\C^\I \to \C^{x/\I}$ is a strict functor, when $\C^{x/\I} $ is structured by gluing $\C^{\strsl x \I}$ with \C along $M_x$ as in \cref{thm:model-gluing}.

  Now we apply \cref{thm:wf} to $\Phi$.
  Thus, assume a \C-inverse category \I and a section of $\Phi$ defined on $\strsl \cinv \I$, i.e.\ that $\C^{\strsl x \I}$ is a cloven model category for all $x\in \I$, and that if $y\prec x$ the functor $\C^{\strsl x \I} \to \C^{y/\I} \cong (\C\dn M_y)$ is strict.
  Our goal is to extend this section to $\cinv/\I$, i.e.\ to construct a cloven model structure on $\C^\I$ such that each $\C^\I \to \C^{x/\I} \cong (\C\dn M_x)$ is strict.
  As noted above, \cref{thm:model-gluing,thm:matching-model-gluing} give model structures on each $\C^{x/\I}$, and 
  by the inductive hypothesis, if $y\prec x$ then the composite $\C^{x/\I} \to \C^{\strsl x\I} \to \C^{y/\I}$ is strict.
  So we have a functor $x\mapsto \C^{x/\I}$ from $\I_0\op$ to cloven model categories and strict functors, whose limit in $\mathrm{Cat}$ is $\C^\I$.
  By \cref{thm:model-limits}, therefore, $\C^\I$ inherits the desired model structure.
\end{proof}

\section{EI \io-categories}
\label{sec:eicats}

Now we specialize further to the case $\C=\sSet$.
In this section we will compare $\sSet$-inverse categories to \io-categories; then in \cref{sec:homotopy-theory} we will extend this to a zigzag of Quillen equivalences relating the model structure of \cref{thm:model-structure} to well-known model structures for \io-presheaves.

\begin{defn}
  An \textbf{EI \io-category} is an \io-category in which every endomorphism is an equivalence.
  There is then an ordering $\prec$ on the equivalence classes of objects, where $x\prec y$ means that there is a noninvertible map $y\to x$.
  An \textbf{inverse EI \io-category} is an EI \io-category such that $\prec$ is well-founded, i.e.\ there are no infinite chains of noninvertible maps $\to\to\to\cdots$.
\end{defn}

In particular, any EI \io-category with finitely many objects is inverse.
An inverse EI \io-category that is a 1-category is still strictly more general than an ordinary inverse category (see e.g.~\cite{bm:extn-reedy}).

We will need to use the following model categories for \io-categories.
\begin{itemize}
\item The Joyal model structure on \sSet for \emph{quasicategories}~\cite{joyal:quasi,lurie:higher-topoi}.
\item The Rezk model structure~\cite{rezk:css} on bisimplicial sets \ssSet for \emph{complete Segal spaces}, and its analogue given by localizing the projective model structure instead of the injective one.
\item The Horel model structure on internal categories in \sSet~\cite{horel:model-intsscat}.
\end{itemize}
These are related by the following Quillen equivalences.
\begin{itemize}
\item The functor $i_1^*:\ssSet \to \sSet$ that takes the 0-simplices at each level is a right Quillen equivalence from the (injective) Rezk model structure to the Joyal model structure~\cite{jt:qcat-segal}.
\item The identity functor is a left Quillen equivalence from the projective Rezk model structure to the injective one.
\item The bisimplicial nerve $N$ of internal categories in \sSet is a right Quillen equivalence from the Horel model structure to the \emph{projective} Rezk model structure.
\end{itemize}

\begin{defn}\label{defn:sigma-i}
  For any \sSet-inverse category \I, define an internal category $\Sigma\I$ by:
  \begin{itemize}
  \item Its object-of-objects is $\Sigma\I_0 = \coprod_{x\in\I_0} \I(x)$.
  \item Its object-of-morphisms is $\Sigma\I_1 = \coprod_{y\prec x} \I(x,y) \sqcup \coprod_{x\in \I_0} \I(x)$.
  \item The source and target maps $\Sigma\I_1 \to \Sigma\I_0$ consist of the projections of the spans $\I(x) \ot \I(x,y) \to \I(y)$ along with the identity on each $\I(x)$.
  \item The identity-assigning map $\Sigma\I_0 \to \Sigma\I_1$ is the inclusion into the second summand.
  \item To define the composition map $\Sigma\I_1 \times_{\Sigma\I_0} \Sigma \I_1 \to \Sigma\I_1$, we observe that by stability of coproducts, its domain decomposes as a coproduct
    \begin{multline*}
      \coprod_{z\prec y\prec x} \I(y,z)\times_{\I(y)} \I(x,y) \sqcup
      \coprod_{y\prec x} \I(y)\times_{\I(y)} \I(x,y) \\ \sqcup
      \coprod_{y\prec x} \I(x,y)\times_{\I(x)} \I(x) \sqcup
      \coprod_x \I(x) \times_{\I(x)} \I(x)
    \end{multline*}
    and so we can put together the composition and identity maps of \I.
  \end{itemize}
\end{defn}

Thus, through the Horel model structure, a \sSet-internal category presents an \io-category.
The following definitions are lifted from~\cite{horel:model-intsscat}.

\begin{defn}\label{defn:ssegal}
  Let \I be a \sSet-inverse category and \K an internal category in \sSet.
  \begin{itemize}
  \item \K is \textbf{strongly Segal} if $\K_0$ is fibrant and the source and target maps $\K_1 \to \K_0$ are fibrations.
  \item \I is \textbf{strongly Segal} if each $\I(x)$ is fibrant and the source and target maps $\I(x,y) \to \I(x)$ and $\I(x,y) \to \I(y)$ are fibrations.
  \item \K is \textbf{Segal-fibrant} if $\K_0$ is fibrant and every pullback
    \( \overbrace{\K_1 \times_{\K_0} \cdots \times_{\K_0} \K_1}^{n} \)
    is fibrant and is a homotopy pullback.
  \item \I is \textbf{Segal-fibrant} if each $\I(x)$ is fibrant and every pullback
    \begin{equation}
      \overbrace{\I(x_{n-1},x_n) \times_{\I(x_{n-1})} \cdots \times_{\I(x_1)} \I(x_0,x_1)}^{n}\label{eq:sfibinv}
    \end{equation}
    is fibrant and is a homotopy pullback.
  \end{itemize}
\end{defn}

\begin{lem}[{\cite[Proposition 5.19]{horel:model-intsscat}}]\label{thm:sseg-sfib}
  Strongly Segal implies Segal-fibrant.\qed
\end{lem}

\begin{lem}\label{thm:SIfib}
  \I is strongly Segal or Segal-fibrant if and only if $\Sigma\I$ is.
\end{lem}
\begin{proof}
  Note that given a family of maps $\{X_i \to Y\}_{i}$ in \sSet, the induced map $\coprod_i X_i \to Y$ is a fibration just when every $X_i \to Y$ is a fibration.
  Similarly, given $\{X_i \to Y_i\}_{i}$, the induced map $\coprod_i X_i \to \coprod_i Y_i$ is a fibration just when every $X_i\to Y_i$ is.
  Moreover, the injections of a coproduct are fibrations.
  It now follows easily that \I is strongly Segal if and only if $\Sigma\I$ is.
  For Segal-fibrancy, as in \cref{thm:SI} we have
  \begin{equation*}
    \overbrace{\K_1 \times_{\K_0} \cdots \times_{\K_0} \K_1}^{n}
    \cong
    \left(\coprod_{x_n \prec \cdots \prec x_0} {\I(x_{n-1},x_n) \times_{\I(x_{n-1})} \cdots \times_{\I(x_1)} \I(x_0,x_1)}\right)
    \sqcup \cdots 
  \end{equation*}
  The omitted summands on the right involve some duplicated objects and some partially trivial pullbacks.
  Thus, all summands are of the form~\eqref{eq:sfibinv} for some possibly smaller $n$, and hence are fibrant and homotopy pullbacks.
  Now the same arguments apply.
\end{proof}

\begin{lem}\label{thm:ssegal}
  If \I is fibrant as in \cref{defn:invcat-fibrant}, it is strongly Segal, hence so is $\Sigma\I$.
\end{lem}
\begin{proof}
  Each $\I(x)$ is fibrant by definition, while
  $\I(x,y) \to \I(x)\times \I(y)$ is the composite of two fibrations $\I(x,y) \to M_y \I(x,\blank) \to \I(x)\times \I(y)$, the first since \I is fibrant and the second by \cref{thm:matching}.
\end{proof}


The pullbacks in the definition of Segal-fibrancy are precisely those occurring in the bisimplicial nerve.
Recall also that Rezk~\cite{rezk:css} defined a bisimplicial set $X$ to be a \emph{Segal space} if it is Reedy fibrant and the induced maps
\begin{equation}
  X_n \to \overbrace{X_1 \times_{X_0} \cdots \times_{X_0} X_1}^n\label{eq:segalspace}
\end{equation}
are all weak equivalences.
These are the fibrant objects in a model structure intermediate between the Reedy/injective one and the complete-Segal-space one.
In the analogous model structure built from the projective one, the fibrant objects are the projective-fibrant ones such that the induced maps to the wide \emph{homotopy} pullback
\begin{equation}
  X_n \to \overbrace{X_1 \times_{X_0}^h \cdots \times_{X_0}^h X_1}^n\label{eq:segalfib}
\end{equation}
are weak equivalences.
Thus, an internal category is Segal-fibrant just when its bisimplicial nerve is fibrant in this projective Segal-space model structure.

\begin{lem}\label{thm:rfrnsfic}
  If \K\ is a Segal-fibrant internal category, then a \emph{Reedy} fibrant replacement of its bisimplicial nerve is a Segal space in the sense of Rezk.
\end{lem}
\begin{proof}
  The property that the maps~\eqref{eq:segalfib} are weak equivalences is invariant under levelwise equivalence, and for Reedy fibrant bisimplicial sets it is equivalent to~\eqref{eq:segalspace} being weak equivalences, since then the actual pullbacks are homotopy pullbacks.
\end{proof}

In particular, if \I is a Segal-fibrant \sSet-inverse category, then a Reedy fibrant replacement of $N\Sigma\I$ is a Segal space.
In fact, more is true:

\begin{lem}\label{thm:reedy=rezk}
  For any Segal-fibrant \sSet-inverse category \I, a Reedy fibrant replacement of $N\Sigma\I$ is Rezk fibrant (i.e.\ a complete Segal space).
\end{lem}
\begin{proof}
  Let $RN\Sigma\I$ be a Reedy fibrant replacement; it remains to prove completeness.
  Since $(N\Sigma\I)_0$ is fibrant, we may assume $(RN\Sigma\I)_0 = (N\Sigma\I)_0$, so $(RN\Sigma\I)_0 = \coprod_{x\in\I_0} \I(x)$.
  Thus, $(RN\Sigma\I)_1$ is a coproduct $\coprod_{x,y\in \I_0} R\I(x,y)$ for some $R\I(x,y)\in\sSet$.
  Since $N\Sigma\I \to RN\Sigma\I$ is a levelwise equivalence, we have $R\I(x,y) = \emptyset$ unless $y\preceq x$.
  Thus, if $y\prec x$, no element of $R\I(x,y)$ can be an equivalence, since there would be nothing to be its inverse.
  So the subspace of components of equivalences in $(RN\Sigma\I)_1$ is contained in $\coprod_{x\in\I_0} R\I(x,x)$.

  Now, because $N\Sigma\I \to RN\Sigma\I$ is a levelwise equivalence, its action on 1-simplices is an equivalence.
  But $(N\Sigma\I)_1 = \coprod_{y\prec x} \I(x,y) \sqcup \coprod_{x} \I(x)$, and the map $(N\Sigma\I)_1 \to (RN\Sigma\I)_1$ sends $\I(x,y)$ into $R\I(x,y)$ and $\I(x)$ into $R\I(x,x)$; thus the induced map $\coprod_{x} \I(x) \to \coprod_{x} R\I(x,x)$ is an equivalence.
  But this is the degeneracy map of $RN\Sigma\I$, so every point in $\coprod_{x\in\I_0} R\I(x,x)$ is an equivalence and $RN\Sigma\I$ is Rezk-complete.
\end{proof}

Thus, any \sSet-inverse category \I gives rise to a complete Segal space $R N \Sigma\I$. 

\begin{thm}
  A fibrant \sSet-inverse category presents an inverse EI \io-category.
\end{thm}
\begin{proof}
  Let \I be a fibrant \sSet-inverse category; by \cref{thm:ssegal,thm:sseg-sfib} it is Segal-fibrant.
  By the proof of \cref{thm:reedy=rezk}, the degeneracy map of $R N \Sigma\I$ is an equivalence onto a subspace that includes all endomorphisms.
  Thus it is EI.
  Since it is Rezk-complete, its set of equivalence classes of objects is the set of connected components of $(R N \Sigma\I)_0$, which is just $\coprod_{x\in\I_0} \I(x)$, and the resulting relation $\prec$ agrees with that induced by the $\prec$ of \I; thus it is well-founded.
\end{proof}

It remains to show that any inverse EI \io-category can be presented by a fibrant \sSet-inverse category.

By~\cite[Proposition 5.13]{horel:model-intsscat}, the fibrant objects of the Horel model structure are created by the bisimplicial nerve to the projective Rezk model structure.
We call them \textbf{Rezk-fibrant}; they are Segal-fibrant and their identity-assigning map $\K_0 \to \K_1$ is an equivalence onto the components of equivalences. 
In particular, $\K_0$ has the homotopy type of the maximal sub-$\infty$-groupoid of \K.
If we write $\K_0$ as a coproduct of connected spaces $\K_0 = \coprod_{x\in\pi_0(\K_0)} \K(x)$, then similarly $\K_1 = \coprod_{x,y\in\pi_0(\K_0)} \K(x,y)$.
Rezk-completeness implies that if $x\neq y$ then no morphism in $\K(x,y)$ can be an equivalence, and each $\K(x) \to \K(x,x)$ is an equivalence onto the components of equivalences.
If \K is an EI \io-category, then every component of $\K(x,x)$ represents an endomorphism and hence an equivalence, so each map $\K(x) \to \K(x,x)$ is an equivalence.
Furthermore, we have $y\prec x$ for $x\neq y$ if and only if $\K(x,y)\neq\emptyset$.

\begin{thm}\label{thm:ei-inverse}
  Any inverse EI \io-category can be presented by one of the form $\Sigma\I$, where $\I$ is a Segal-fibrant \sSet-inverse category.
\end{thm}
\begin{proof}
  Using the Horel model structure, any small \io-category may be presented by a Rezk-fibrant internal category $\K$ in \sSet.
  When \K is EI, as we shall henceforth assume,
  the above arguments show that $\K_0 = \coprod_{x\in\I_0} \K(x)$ and $\K_1 = \coprod_{y\preceq x} \K(x,y)$, and the maps $\K(x) \to \K(x,x)$ are equivalences.

  Define a \sSet-inverse category \I with $\I_0 = \pi_0(\K_0)$ and $\prec$ that of \K, with $\I(x) = \K(x)$, $\I(x,y) = \K(x,y)$ for $y\prec x$ and composition induced from \K.
  Segal-fibrancy of $\Sigma\I$ follows from that of \K.
  We will show that the obvious functor $\Sigma\I \to \K$ is an equivalence in the projective model structure of~\cite[Theorem 5.2]{horel:model-intsscat}, hence also the Horel model structure.
  This means we must show that it induces a levelwise equivalence of bisimplicial nerves.
  It is an isomorphism on spaces of objects, while on morphisms it is a coproduct of the equalities $\I(x,y) = \K(x,y)$ when $y\prec x$ and the above equivalences $\K(x) \to \K(x,x)$.
  Thus it remains to show that the map
  \[ \overbrace{\Sigma\I_1 \times_{\Sigma\I_0} \cdots \times_{\Sigma\I_0} \Sigma\I_1}^{n} \too
  \overbrace{\K_1 \times_{\K_0} \cdots \times_{\K_0} \K_1}^{n}
  \]
  is an equivalence for all $n\ge 2$.
  Now this map lies over
  \[(\K_0)^{n+1} \cong \coprod_{x_0,\dots,x_n} \K(x_0)\times\cdots\times\K(x_n),\]
  so it will suffice to show that each induced map
  \begin{multline}
    \Sigma\I(x_0,x_1) \times_{\K(x_1)} \cdots \times_{\K(x_{n-1})} \Sigma\I(x_{n-1},x_n) \\
    \too \K(x_0,x_1) \times_{\K(x_1)} \cdots \times_{\K(x_{n-1})} \K(x_{n-1},x_n)\label{eq:SIeqn}
  \end{multline}
  is an equivalence.
  Here $\Sigma\I(x,y)$ denotes the summand of $\Sigma\I_1$ lying over $\I(x)\times\I(y)$, which is $\K(x,y)$ if $y\prec x$, is $\K(x)$ if $x=y$, and is $\emptyset$ otherwise.
  Thus, $\Sigma\I(x,y) \to \K(x,y)$ is an identity if $y\prec x$, an equivalence if $x=y$, and an identity otherwise.

  In particular,~\eqref{eq:SIeqn} is an isomorphism unless some $x_i$ are duplicated.
  If there are duplications, the domain of~\eqref{eq:SIeqn} is a wide pullback like the codomain, but for the shorter list of $x_i$'s obtained by omitting adjacent duplicates.
  Since $\K$ is Segal-fibrant, each of these wide pullbacks is a homotopy pullback.
  But homotopy pullbacks preserve equivalences, and the maps $\K(x) \to \K(x,x)$ are equivalences.
\end{proof}

It remains to replace a Segal-fibrant \sSet-inverse category by a fibrant one.

\begin{defn}
  For an internal category \K, an internal diagram $A\in(\sSet^\K)^\Gm$ is \textbf{Segal-fibrant} if each wide pullback
  \( A \times_{\K_0} \overbrace{\K_1 \times_{\K_0} \cdots \times_{\K_0} \K_1}^{n} \)
  is fibrant and is a homotopy pullback.
\end{defn}

In particular, \K is Segal-fibrant iff $\K_1$ is Segal-fibrant as an object of $(\sSet^\K)^{\K_0}$.

\begin{lem}\label{thm:bar}
  Let $f:\K\to\L$ be a functor between internal categories in \sSet, let $A \in (\sSet^\K)^\Gm$, and assume that
  $A$ is Segal-fibrant,
  \L is strongly Segal,
  $f_0:\K_0 \to \L_0$ is an isomorphism, and
  $f_1:\K_1 \to \L_1$ is a weak equivalence.
  Then there is a $B\in (\sSet^\L)^\Gm$ and a map $A\to f^*B$ in $(\sSet^\K)^\Gm$ whose underlying map in $\sSet$ is a weak equivalence.
\end{lem}
\begin{proof}
  We mimic~\cite[Theorem 6.22]{horel:model-intsscat}.
  Since $f_0$ is an isomorphism, $f^*B$ is $B$ with a \K-action induced by $f$.
  Let $B$ be the bar construction $B(A,\K,\L)$ as in~\cite{may:csf}; then we have a simplicial homotopy equivalence $A\to B(A,\K,\K)$, so it suffices to show the map $B(A,\K,\K) \to B(A,\K,\L)$ induced by $f_1$ is a weak equivalence.

  Since geometric realization preserves weak equivalences, it suffices to show each
  \[ A \times_{\K_0} \overbrace{\K_1 \times_{\K_0} \cdots \times_{\K_0} \K_1}^{n} \times_{\K_0} \K_1 \too
  A \times_{\K_0} \overbrace{\K_1 \times_{\K_0} \cdots \times_{\K_0} \K_1}^{n} \times_{\K_0} \L_1
  \]
  is a weak equivalence.
  This is because both pullbacks are homotopy pullbacks, by Segal-fibrancy of $A$ and strong-Segality of \L, and $f_1$ is a weak equivalence.
\end{proof}


\begingroup
\let\C\sSet
\begin{defn}
  For \C-inverse categories \I and \J, an \textbf{io-functor} $\I\to\J$ is
  \begin{itemize}
  \item An injection $\I_0 \into \J_0$ that is the inclusion of an initial segment,
  \item Isomorphisms $\I(x) \cong \J(x)$ for all $x\in \I_0$, and
  \item Morphisms $\I(x,y) \to \J(x,y)$ over the isomorphism $\I(x)\times\I(y) \cong \J(x)\times \J(y)$, commuting with composition.
  \end{itemize}
  An \textbf{io-embedding} is an io-functor such that each $\I(x,y) \to \J(x,y)$ is an isomorphism.
  An \textbf{io-equivalence} is an io-functor such that $\I_0 \into \J_0$ is an isomorphism and each $\I(x,y) \to \J(x,y)$ is a weak equivalence.
\end{defn}

For example, the inclusion of any full \C-inverse subcategory is an io-embedding.
Any io-functor $f:\I\to\J$ induces an ordinary internal functor $\Sigma f:\Sigma\I\to\Sigma\J$.
\endgroup

\begin{lem}
  If $f:\I\to\J$ is an io-equivalence and $\I$ and $\J$ are Segal-fibrant, then $\Sigma f$ is a weak equivalence in the Horel model structure.
\end{lem}
\begin{proof}
  In fact, it is a projective equivalence, i.e.\ induces a levelwise equivalence of bisimplicial nerves.
  The induced map on $n$-simplices is a coproduct of maps
  \[ \I(x_{n-1},x_n) \times_{\I(x_{n-1})} \cdots \times_{\I(x_1)} \I(x_0,x_1) \too
  \J(x_{n-1},x_n) \times_{\J(x_{n-1})} \cdots \times_{\J(x_1)} \J(x_0,x_1),
  \]
  between homotopy pullbacks, hence preserving the equivalences $\I(x,y) \to \J(x,y)$.
\end{proof}

\begin{lem}\label{thm:ibar}
  Let $f:\I\to\J$ be an io-equivalence and let $A \in (\sSet^\I)^\Gm$, where $A$ is Segal-fibrant and \J is fibrant.
  Then there is a Reedy fibrant $B\in (\sSet^\J)^\Gm$ and a weak equivalence $A\to f^*B$ in $(\sSet^\I)^\Gm$.
\end{lem}
\begin{proof}
  Apply \cref{thm:bar} to $\Sigma f$ to obtain $B'\in (\sSet^\J)^\Gm$ with a weak equivalence $A\to f^*B'$, and then let $B$ be a Reedy fibrant replacement of $B'$.
  Since $f^*$ doesn't change the underlying objects, it preserves weak equivalences, so the composite $A\to f^*B' \to f^* B$ is again a weak equivalence.
\end{proof}

\begingroup
\let\C\sSet
\begin{thm}\label{thm:invcat-fibrepl}
  For any Segal-fibrant \sSet-inverse category \I, there is a fibrant \sSet-inverse category \Ibar and an io-equivalence $\I\to\Ibar$.
\end{thm}
\begin{proof}
  The idea is that we can extend an io-equivalence $f_{\strsl x\I} : (\strsl x\I)\to \overline{\strsl x \I}$ with $\overline{\strsl x \I}$ fibrant to an io-equivalence $f_{x/\I}:(x/\I) \to \overline{x/\I}$ with $\overline{x/\I}$ fibrant, where $\overline{x/\I}(x,\blank)$ is obtained by applying \cref{thm:ibar} to $f_{\strsl x \I}$ and $\I(x,\blank)$.
  But there are technical details needed to make the well-founded recursion precise by applying \cref{thm:wf}.
  
  First, fix a particular function implementing \cref{thm:ibar}: thus it assigns to every io-equivalence $f:\I\to\J$, with \J fibrant, and Segal-fibrant $A \in (\sSet^\I)^\Gm$, a Reedy fibrant object $\Theta(f,A)\in (\sSet^\J)^\Gm$ and a weak equivalence $\theta_{f,A} : A\to f^* \Theta(f,A)$.
  Let \ssetinvs be the subclass of \ssetinv consisting of the Segal-fibrant \sSet-inverse categories, which inherits a well-founded relation from \ssetinv.
  We define $\Phi:\cZ \to \ssetinvs$ as follows.
  An object of $\cZ$ over \I is an io-equivalence $f_\I : \I\to \Ibar$ where \Ibar is fibrant.
  A morphism of \cZ over $(\strsl x \I)\prec \I$ is an io-embedding $\overline{\strsl x\I} \to \Ibar$ such that
  \begin{equation}\label{eq:fibrepl-sq}
    \vcenter{\xymatrix@-.5pc{
        \strsl x\I\ar[r]\ar[d]_{f_{\strsl x \I}} &
        \I\ar[d]^{f_\I}\\
        \overline{\strsl x\I}\ar[r] &
        \Ibar
      }}
  \end{equation}
  commutes, with an isomorphism $\Theta(f_{\strsl x \I},\I(x,\blank)) \cong \Ibar(x,\blank)$ such that the composite
  \begin{equation}
    \I(x,\blank) \xto{\theta_{f_{\strsl x \I},\I(x,\blank)}} f_{\strsl x \I}^*\Theta(f_{\strsl x \I},\I(x,\blank)) \toiso f_{\strsl x \I}^*\Ibar(x,\blank)\label{eq:fibrepl-thetaf}
  \end{equation}
  is equal to the action of $f_\I$.
  Note that $\I(x,\blank)$ is Segal-fibrant since \I is, while $\overline{\strsl x\I}$ is fibrant by assumption, so this $\Theta$ is valid.
  To compose morphisms of \cZ over $(\strsl y {(\strsl x \I)}) \prec (\strsl x \I)\prec \I$, we take the isomorphism to be the composite
  \begin{equation}
    \Theta(f_{\strsl y \I},\I(y,\blank)) \cong \overline{(\strsl x \I)}(y,\blank) \cong \Ibar(y,\blank)\label{eq:fibrepl-theta-comp}
  \end{equation}
  in which the second isomorphism comes from the fact that $\overline{\strsl x\I} \to \Ibar$ is an io-embedding.

  Applying \cref{thm:wf}, we may assume a partial section of $\Phi$ defined on $\strsl \ssetinvs \I$.
  Thus, we have io-equivalences $f_{\strsl x\I} : (\strsl x\I)\to \overline{\strsl x \I}$ for all $x\in\I$, with each $\overline{\strsl x \I}$ fibrant, and io-embeddings $\overline{\strsl y\I} \to \overline{\strsl x\I}$ for $y\prec x$ giving commutative squares
  \begin{equation*}
    \vcenter{\xymatrix@-.5pc{
        \strsl y\I\ar[r]\ar[d]_{f_{\strsl y \I}} &
        \strsl x \I\ar[d]^{f_{\strsl x \I}}\\
        \overline{\strsl y\I}\ar[r] &
        \overline{\strsl x\I}.
      }}
  \end{equation*}
  We also have $\Theta(f_{\strsl y \I},\I(y,\blank)) \cong \overline{(\strsl x \I)}(y,\blank)$ (using the fact that $(\strsl x \I)(y,\blank) = \I(y,\blank)$ by definition) such that the composite 
  \[ (\strsl x \I)(y,\blank) \to f_{\strsl y \I}^*\Theta(f_{\strsl y \I},\I(y,\blank)) \toiso f_{\strsl y \I}^*\overline{(\strsl x \I)}(y,\blank) \]
  is equal to the action of $f_{\strsl x \I}$.
  Moreover, when $z\prec y \prec x$, the composite 
  \[ \Theta(f_{\strsl z \I},\I(z,\blank)) \cong \overline{(\strsl y \I)}(z,\blank) \cong \overline{(\strsl x \I)}(z,\blank) \]
  is equal to the specified isomorphism $\Theta(f_{\strsl z \I},\I(z,\blank)) \cong \overline{\strsl x \I}(z,\blank)$.

  Our goal is to construct a fibrant \Ibar and an io-equivalence $f_\I : \I\to \Ibar$, along with io-embeddings $\overline{\strsl x \I} \to\Ibar$ giving~\eqref{eq:fibrepl-sq}, and isomorphisms $\Theta(f_{\strsl x \I},\I(x,\blank)) \cong \Ibar(x,\blank)$ such that~\eqref{eq:fibrepl-thetaf} equals $f_\I$, and whenever $y\prec x$,~\eqref{eq:fibrepl-theta-comp} is equal to the given $\Theta(f_{\strsl y \I},\I(y,\blank)) \cong \Ibar(y,\blank)$.
  First, define $\overline{x/\I}(x,\blank) = \Theta(f_{\strsl x \I},\I(x,\blank))$ for each $x\in \I$; by the argument in \cref{thm:collage}, this yields a fibrant $\overline{x/\I}$ with an io-equivalence $f_{x/\I}:x/\I \to \overline{x/\I}$ and an io-embedding $\overline{\strsl x \I} \to \overline{x/\I}$.
  Our inductive assumption implies each io-embedding $\overline{\strsl y \I} \into \overline{\strsl x \I}$ factors through $\overline{y/\I}$ by io-embeddings, so we have a composite io-embedding
  \(\overline{y/\I} \into \overline{\strsl x \I} \into \overline{x/\I}\).
  Each composite $\overline{z/\I} \to \overline{y/\I} \to \overline{x/\I}$ is equal to $\overline{z/\I} \to \overline{x/\I}$ by the inductive functoriality assumption, and the following diagrams commute by construction:
  \begin{equation}
    \vcenter{\xymatrix@-.5pc{
        y/\I\ar[r]\ar[d] &
        \strsl x\I\ar[r]\ar[d] &
        x/\I\ar[d]\\
        \overline{y/\I}\ar[r] &
        \overline{\strsl x \I}\ar[r] &
        \overline{x/\I}
      }}
    \qquad
    \vcenter{\xymatrix@-.5pc{
        \overline{\strsl y \I}\ar[r]\ar[d] &
        \overline{\strsl x \I}\ar[d]\\
        \overline{y/\I}\ar[r] &
        \overline{x/\I}
      }}\label{eq:fibrepl-xIcomm}
  \end{equation}

  Now we have a functor from $\I_0$ to the category of \C-inverse categories and io-embeddings which send $x$ to $\overline{x/\I}$.
  We define $\Ibar(x,y) = \overline{x/\I}(x,y)$; we can compose these since each $\overline{y/\I} \into \overline{x/\I}$ is an io-embedding, and the functoriality of these io-embeddings gives associativity.
  Since each $\overline{x/\I}$ is fibrant, so is $\Ibar$, and we have io-embeddings $\overline{x/\I}\into \Ibar$ giving commutative triangles as on the left below.
  \begin{equation*}
    \vcenter{\xymatrix@-.5pc{\overline{y/\I} \ar[r] \ar[d] & \Ibar \\ \overline{x/\I} \ar[ur] }}
    \qquad
    \vcenter{\xymatrix@-.5pc{
        x/\I\ar[r]\ar[d] &
        \I\ar[d]\\
        \overline{x/\I}\ar[r] &
        \Ibar.
      }}
  \end{equation*}
  Similarly, the io-equivalences $f_{x/\I}:x/\I \to \overline{x/\I}$ assemble into an io-equivalence $f_\I:\I\to \Ibar$ making the square on the above right commute.
  With~\eqref{eq:fibrepl-xIcomm}, this yields~\eqref{eq:fibrepl-sq}.
  The rest of the necessary properties follow directly from the definition of \Ibar.
\end{proof}
\endgroup

\begin{cor}\label{thm:ei-fibinv}
  Any inverse EI \io-category can be presented by an internal category in \sSet of the form $\Sigma\I$, where \I is a fibrant \sSet-inverse category.\qed
\end{cor}

\section{\io-presheaves}
\label{sec:homotopy-theory}

Our goal now is to compare the model structure of \cref{thm:model-structure} for $\C=\sSet$ to a standard presentation of \io-presheaves.
We begin with the following observations.


\begingroup
\let\C\sSet
\let\lC\lsSet
\begin{lem}\label{thm:SI}
  For any \C-inverse category \I, the \C-indexed category $\lC^\I$ from \cref{defn:idiag} is equivalent to the ordinary \C-indexed diagram category $\lC^{\Sigma \I}$.
\end{lem}
\begin{proof}
  Extensivity of \C implies (see~\cite{clw:ext-dist}) that a morphism $A\to X\times \Sigma\I_0$ is uniquely determined by a family of objects $\{A_x\}_{x\in\I_0}$ with morphisms $A_x \to X\times \I(x)$.
  Similarly, an action of $\Sigma\I_1$ on $A$ decomposes into actions $A_x \times_{\I(x)} \I(x,y) \to A_y$.
\end{proof}
\endgroup

\begin{thm}[{\cite[Proposition 6.6]{horel:model-intsscat}}]\label{thm:proj-model}
  If \K is a strongly Segal internal category in \sSet, then there is a \textbf{projective model structure} on $\sSet^\K$ whose fibrations and weak equivalences are created by the forgetful functor $\sSet^\K \to \sSet/\K_0$.
\end{thm}

Recall that by \cref{thm:ssegal}, $\Sigma\I$ is strongly Segal whenever \I is fibrant.

\begin{lem}\label{thm:reedy=projective}
  For a fibrant $\sSet$-inverse category \I, the equivalence of categories from \cref{thm:SI} is a right Quillen equivalence from the Reedy model structure on $\sSet^\I$ of \cref{thm:model-structure} to the projective model structure on $\sSet^{\Sigma\I}$.
\end{lem}
\begin{proof}
  Since $\coprod_i X_i \to \coprod_i Y_i$ is a fibration or weak equivalence if and only if each $X_i \to Y_i$ is, the Reedy weak equivalences in $\sSet^\I$ coincide with the projective ones in $\sSet^{\Sigma\I}$.
  And by \cref{thm:fib-objwise}, every Reedy fibration in $\sSet^\I$ is an objectwise fibration, hence a projective fibration in $\sSet^{\Sigma\I}$.
\end{proof}

\begin{thm}[{\cite{pbb:groth-segal}}]
  If \K\ is a strongly Segal internal category in \sSet, there is a quasicategory $X$ presenting the same \io-category \K and a zigzag of Quillen equivalences from the projective model structure on $\sSet^\K$ to the left fibration model structure on $\sSet/X$ (as studied in~\cite[\S2.1]{lurie:higher-topoi}).
\end{thm}
\begin{proof}
  Let $X = i_1^* R N \K$, where $N$ denotes the bisimplicial nerve, $R$ denotes complete-Segal-space fibrant replacement, and $i_1^*$ takes complete Segal spaces to quasicategories as in~\cite{jt:qcat-segal}.
  Since $N$ and $R$ preserve all weak equivalences, and $i_1^*$ is a right Quillen equivalence, $X$ presents the same \io-category as $\K$.
  The desired zigzag is
  \begin{align*}
    (\sSet^\K)_{\mathrm{proj}}
    &\leftrightarrows (\ssSet/N\K)_{\mathrm{left,proj}}
    \tag*{\cite[Theorem 1.40]{pbb:groth-segal}}\\
    &\rightleftarrows (\ssSet/RN\K)_{\mathrm{left,proj}}
    \tag*{\cite[Corollary 5.7]{pbb:groth-segal}}\\
    &\rightleftarrows (\ssSet/RN\K)_{\mathrm{left,inj}}\\
    &\leftrightarrows (\sSet/i_1^* RN\K)_{\mathrm{left}}
    \tag*{\cite[Theorem 1.22]{pbb:groth-segal}}
  \end{align*}
  Here $(\ssSet/B)_{\mathrm{left,proj}}$ and $(\ssSet/B)_{\mathrm{left,inj}}$ are the projective and injective versions of the \emph{left fibration model structure} over a Segal space $B$ from~\cite[Proposition 1.10]{pbb:groth-segal}.
  The unlabeled equivalence is an identity functor, which is a Quillen equivalence.
\end{proof}

\begin{cor}\label{thm:reedy-io}
  For a fibrant \sSet-internal category \I, there is a zig-zag of Quillen equivalences relating the Reedy model structure on $\sSet^\I$ with a model category
  presenting the \io-category of diagrams over the \io-category presented by $\Sigma\I$.\qed
\end{cor}

\begin{cor}
  The \io-category of diagrams on any inverse EI \io-category can be presented by the Reedy model structure on $\sSet^\I$ for some fibrant \sSet-inverse category \I.\qed
\end{cor}

\section{Type-theoretic fibration categories}
\label{sec:type-theory}

Type-theoretic fibration categories were defined in~\cite{shulman:invdia} to abstract
the categorical structure that interprets type theory. 
The intent 
was to 
emphasize the homotopy-theoretic point of view that they are particular categories of fibrant objects~\cite{brown:ahtgsc}.

\begin{defn}\label{def:ttfc}
  A \textbf{type-theoretic fibration category} is a category \C with:
  \begin{enumerate}[leftmargin=*,label=(\arabic*)]
  \item A terminal object $1$.\label{item:cat1}
  \item A subcategory of \textbf{fibrations} containing all the isomorphisms and all the morphisms with codomain $1$.\label{item:cat2a}
    A morphism is called an \textbf{acyclic cofibration} if it has the left lifting property with respect to all fibrations.
  \end{enumerate}
  such that
  \begin{enumerate}[leftmargin=*,label=(\arabic*),resume]
  \item All pullbacks of fibrations 
    exist and are fibrations.\label{item:cat3}
  \item The dependent product of a fibration along a fibration exists and is again a fibration.
    Thus, acyclic cofibrations are stable under pullback along fibrations.\label{item:cat4}
  \item Every morphism factors as an acyclic cofibration followed by a fibration.\label{item:cat7}
  \end{enumerate}
\end{defn}

In~\cite{shulman:invdia} the following property was included in the definition, but Joyal has pointed out that it follows from the other axioms.

\begin{lem}\label{thm:cat8}
  If $g:B\to C$ and $gi: A\to C$ are fibrations, $i: A\to B$ is an acyclic cofibration, and both squares below are pullbacks (hence $f:Y\to Z$ and $fj:X\to Z$ are fibrations by~\ref{item:cat3}), then $j:X\to Y$ is also an acyclic cofibration.
  \[\vcenter{\xymatrix{
      X\ar[r]^{j}\ar[d]^r \pullback &
      Y\ar[r]^f\ar[d]^q \pullback &
      Z\ar[d]^p\\
      A\ar[r]_i &
      B\ar[r]_g &
      C
    }}\]
\end{lem}
\begin{proof}
  If $p$ is a fibration, so is $q$, hence pullback along $q$ preserves acyclic cofibrations.
  Thus, factoring $p$, we may assume it is an acyclic cofibration.
  Since $g$ and $gi$ are fibrations, $q$ and $r$ are acyclic cofibrations, hence so is $ir = qj$.
  We conclude by:\phantom{\qedhere}
\end{proof}

\begin{lem}[von Glehn]
  If $gf$ and $g$ are acyclic cofibrations, so is $f$.
\end{lem}
\begin{proof}
  Suppose $tf = ps$, with $p:A\to B$ a fibration.
  Since $B\to 1$ is a fibration and $g$ an acyclic cofibration, we have $h$ with $hg=t$.
  Since $gf$ is an acyclic cofibration, we have $k$ with $p k = h$ and $k (g f) = s$.
  Therefore, $kg$ satisfies $(kg)f=s$ and $p(kg)=hg=t$.
\end{proof}

The main theorem from~\cite{shulman:invdia} we will use is the preservation of type-theoretic fibration categories under gluing, i.e.\ certain comma categories.
The functors we can glue along are these:

\begin{defn}\label{def:fibfr}
 A functor between type-theoretic fibration categories is a \textbf{strong fibration functor} if it preserves terminal objects, fibrations, pullbacks of fibrations, and homotopy equivalences.
\end{defn}

Here the ``homotopy equivalences'' are defined using the path objects constructed from the fibration structure.

\begin{thm}[{\cite{shulman:invdia}}]\label{thm:gluing}
  If \C and \D are type-theoretic fibration categories and $G:\C\to\D$ is a strong fibration functor, then the category $(\D\dn G)\f$, equipped with the Reedy fibrations, is a type-theoretic fibration category.
  If \C and \D contain universe objects satisfying the univalence axiom (see~\cite{shulman:invdia}), so does $(\D\dn G)\f$.
  Moreover, the forgetful functor $(\D\dn G)\f \to \C$ preserves all of the structure strictly.
\end{thm}

Here $(\D\dn G)\f$ is the subcategory of {Reedy fibrant objects} in the comma category $(\D\dn G)$;
a morphism $A\to B$ in $(\D\dn G)$ is a \textbf{Reedy fibration} if $A_0 \to B_0$ is a fibration in \C and the induced map $A_1 \to G A_0 \times_{G B_0} B_1$ is a fibration in \D, and $A$ is \textbf{Reedy fibrant} if $A\to 1$ is a Reedy fibration.
In~\cite{shulman:invdia} I also assumed that $G$ preserves acyclic cofibrations; for a sketch of how the proof needs to be modified without this assumption, see \cref{sec:fibfr}.

Let \C be a type-theoretic fibration category; we will show that diagrams on \C-inverse categories are also type-theoretic fibration categories, hence model type theory.
We apply \cref{sec:category-theory} by using the fibrations of \C as \emph{both} the fibrations and the prefibrations.

\begin{defn}[\cite{shulman:invdia}]\label{def:Ilim}
  For $I$ a well-founded poset, \C has \textbf{Reedy $I\op$-limits} if
  \begin{enumerate}
  \item Any Reedy fibrant $A\in\sC^{I\op}$ has a limit, which is fibrant in \sC.\label{item:il1}
  \end{enumerate}
  and for Reedy fibrant $A$ and $B$ and any morphism $f: A\to B$, the following hold:
  \begin{enumerate}[resume]
  \item If $f$ is a Reedy fibration, then $\lim f: \lim A \to \lim B$ is a fibration in \sC.\label{item:il2}
  \item If $f$ is a levelwise equivalence, then $\lim f$ is an equivalence in \sC.\label{item:il3}
  \item If $f$ is a Reedy acyclic cofibration, then $\lim f$ is an acyclic cofibration in \sC.\label{item:il4}
  \end{enumerate}
  We say $I$ is \textbf{admissible} for \C if \C has Reedy $(\strsl I x)\op$-limits for all $x\in I$.
\end{defn}

\begin{defn}\label{defn:admissible}
  A \C-inverse category \I is \textbf{admissible} if $\I_0$ is admissible for \C (hence also pre-admissible as in \cref{defn:pre-admissible}).
\end{defn}

This is automatic if each $\strsl x \I_0\op$ is finite or if \C is a type-theoretic model category.

Recall from~\cite[Lemma 5.9]{shulman:invdia} that \C satisfies \emph{function extensionality} if and only if dependent products along fibrations preserve acyclicity of fibrations.

\begin{lem}\label{thm:homs-eqv}
  Suppose \C satisfies function extensionality, \I is a fibrant \C-inverse category, $A\in (\lC^\I)^X$ and $B,B'\in (\lC^\I)^Y$ are Reedy fibrant, and $\lC^{X\times Y}$ has Reedy $\I_0\op$-limits.
  If $g:B \to B'$ is a homotopy equivalence, so is $\lC^\I(A,B) \to \lC^\I(A,B')$.
\end{lem}
\begin{proof}
  We modify the proof of \cref{thm:homs-fibration} slightly.
  Since Reedy limits preserve equivalences between Reedy fibrant objects, for the inductive step it suffices to show each $\lC^{x/\I}(A,g): \lC^{x/\I}(A,B) \to \lC^{x/\I}(A,B')$ is an equivalence if $g$ is.
  This is the top morphism in the following square, in which we have also included the pullback:
  \begin{center}
    \begin{tikzcd}[row sep=3em,column sep=huge]
      (\pi_2)_* {\left((c^*B_x)^{\pi_1^*A_x}\right)} \ar[r,"{\lC^{x/\I}(A,g)}"]\ar[d]
      \ar[dr,phantom,"\bullet"{name=PB},near start]
      \ar[dashed,to=PB,"r"']
      &
      (\pi_2')_* {\left(({c'}^*B'_x)^{{\pi'_1}^*A_x}\right)} \ar[d]
      \ar[leftarrow,to=PB,"p"] \\
      \lC^{\strsl x \I}(A,B) \ar[r,"{\lC^{\strsl x \I}(A,g)}"'] 
      \ar[leftarrow,to=PB]&
      \lC^{\strsl x \I}(A,B').
    \end{tikzcd}
  \end{center}
  By the inductive hypothesis, the bottom morphism $\lC^{\strsl x \I}(A,g)$ is an equivalence.
  Since equivalences are stable under pullback along fibrations, the map $p$ is also an equivalence.
  Thus it suffices to show that $r$ is an equivalence.

  However, $r$ is the same morphism that in \cref{thm:homs-fibration} we were showing to be a fibration.
  Since pullback, dependent products along fibrations, and local exponentials by fibrations all preserve equivalences between fibrations (using function extensionality in the latter two cases), we can use the same argument as in \cref{thm:homs-fibration} once we know that $B_x \to M_x B \times_{M_x B'} B'_x$ is an equivalence.
  By 2-out-of-3, this follows from $B_x \to B_x'$ and $M_x B \times_{M_x B'} B'_x \to B'_x$ being equivalences:
  the first by assumption, and the second as a pullback of $M_x B \to M_x B'$ (an equivalence by the inductive hypothesis) along the fibration $B'_x \to M_xB'$.
\end{proof}


Recall that $\C^\I$ means $(\lC^\I)^1$; let $\C^\I\f$ be its full subcategory of Reedy fibrant objects.

\begin{lem}\label{thm:matching-fibfr}
  Suppose \C satisfies function extensionality, \I is fibrant and admissible, and that for some $x\in \I$, the Reedy fibrations make $\C^{\strsl x \I}\f$ into a type-theoretic fibration category.
  Then the functor $M_x : \C^{\strsl x \I}\f \to \C$ is a strong fibration functor.
\end{lem}
\begin{proof}
  By \cref{thm:matching-fibration,thm:homs-eqv} it preserves fibrations and equivalences,
  and hom-functors $\lD(A,-)$ preserve all limits that the reindexing functors of \lD do.
\end{proof}

Thus, under the hypotheses of \cref{thm:matching-fibfr},~\cite[\S 13]{shulman:invdia} implies that $(\C\dn M_x)\f$ is a type-theoretic fibration category; while \cref{thm:matching-gluing} says that this category is equivalent to $\C^{x/\I}\f$.
This is the crucial step in the following theorem.


\begin{thm}\label{thm:ttfc}
  Suppose \C is a type-theoretic fibration category satisfying function extensionality, and \I is a fibrant and admissible \C-inverse category.
  Then the Reedy fibrations make $\C^\I\f$ into a type-theoretic fibration category, which has as many nested univalent universes as \C does.
  Moreover, if \C is cloven~\cite[Definition \ref{invdia:def:cttfc}]{shulman:invdia} or split~\cite[{Definition \ref{invdia:def:split}}]{shulman:invdia}, then so is $\C^\I\f$.
\end{thm}
\begin{proof}
  Suppose \C is cloven (otherwise, cleave it).
  We argue by well-founded induction as in \cref{thm:model-structure}, defining a section of the following functor $\Phi:\cZ\to \cinv$.
  An object of \cZ over $\I\in\cinv$ is a cloven type-theoretic fibration category structure on $\C^\I\f$ with as many univalent universes as \C.
  A morphism of \cZ over $(\strsl x\I)\prec \I$ is the assertion that $\C^\I\f \to \C^{x/\I}\f$ is a strict functor, when $\C^{x/\I}\f \cong (\C\dn M_x)\f$ is structured by gluing $\C^{\strsl x \I}\f$ with \C.

  Applying \cref{thm:wf} to construct a section of $\Phi$, we assume given a \C-inverse category \I and a section of $\Phi$ defined on $\strsl \cinv \I$, i.e.\ that $\C^{\strsl x \I}\f$ is a cloven type-theoretic fibration category for all $x\in \I$, and that if $y\prec x$ the functor $\C^{\strsl x \I}\f \to \C^{y/\I}\f \cong (\C\dn M_y)\f$ is strict.
  By \cref{thm:matching-fibfr,thm:gluing,thm:reedy-char}, each $\C^{x/\I}\f$ inherits such a structure, and the composite $\C^{x/\I}\f \to \C^{\strsl x\I}\f \to \C^{y/\I}\f$ is strict.
  We must extend this section to $\cinv/\I$, i.e.\ construct such a structure on $\C^\I\f$ such that each $\C^\I\f \to \C^{x/\I}\f \cong (\C\dn M_x)\f$ is strict.
  But the above structures on the categories $\C^{x/\I}\f$ yield a functor from $\I_0\op$ to \ttfc, whose limit in $\mathrm{Cat}$ is $\C^\I\f$.
  Thus its limit in \ttfc gives the desired structure on $\C^\I\f$.
\end{proof}


Finally, we specialize to the case when $\C=\sSet\f$.

\begin{cor}\label{thm:sset-invdia}
  For any fibrant $\sSet$-inverse category \I, the category $(\sSet\f)^\I$ supports a model of type theory\footnote{But see footnote~\ref{fn:initiality}.} with a unit type, dependent sums and products, identity types, and as many univalent universes as there are inaccessible cardinals.
\end{cor}

Moreover, in this case, the type-theoretic fibration category $(\sSet\f)^\I$ arising from \cref{thm:ttfc} coincides with the underlying type-theoretic fibration category $(\sSet^\I)\f$ of fibrant objects in the model category $\sSet^\I$ from \cref{thm:model-structure}.
And since all objects of $\sSet^\I$ are cofibrant, the right homotopy equivalences in $(\sSet^\I)\f$ coincide with the model-categorical weak equivalences, so the two present the same \io-category.
(This condition, though sometimes omitted, is necessary; see~\cite{bordg:thesis}.)
Thus, the model of type theory from \cref{thm:sset-invdia} may be said to live in the \io-category presented by the model category $\sSet^\I$.
Combining this with \cref{thm:ei-fibinv}, we have:

\begin{cor}\label{thm:ei-tt}
  For any inverse EI \io-category \K, the \io-category $\ig^\K$ supports a model of type theory 
  with a unit type, dependent sums and products, identity types, and with as many univalent universes as there are inaccessible cardinals.\qed
\end{cor}

\section{Fibrant internal inverse categories}
\label{sec:fibrant-categories}

To end the paper, we will describe more explicitly in some small examples what it means for a \C-inverse category to be ``fibrant'' in the sense of \cref{defn:invcat-fibrant}, and what the corresponding Reedy fibrant diagrams are.
We will express these in terms of the internal type theory of \C, so we begin with a brief review of this.

The \emph{types} in type theory correspond to fibrant objects, or more generally fibrations, in a category.
A type can depend on variables in some other type, e.g.\ if $x:A$ we might have a type $B(x)$ depending on $x$; this corresponds to having a fibrant object $A$ and a fibration $B \to A$, with each $B(x)$ representing its ``fiber over $x$''. 
In this case we can form its \emph{dependent sum} $\sm{x:A} B(x)$, which is the domain $B$ of the fibration, and also its \emph{dependent product} $\prd{x:A} B(x)$, which is obtained from the right adjoint to pullback along $A\to 1$. 

More formally, type theory consists of ``judgments'' that look like $\Gm \vdash p:P$ or $\Gm\vdash P\ty$, where $\Gm$ is a \emph{context} consisting of a list of variables assigned to types, each type perhaps depending on the previous ones.
For example, $(x:A), (y:B(x)), (z:C(x,y))$ is a context containing three variables.
Such a context represents categorically a tower of fibrations such as $C\to B\to A\to 1$.
A judgment $\Gm\vdash P\ty$ represents a further fibration $P\to C$ over the top object in this tower, and $\Gm \vdash p:P$ means that this fibration has a section. 
If some or all of the variables in $\Gm$ don't appear in $P$, that means this fibration was pulled back to $C$ from some earlier stage in the tower. 

\subsection{No objects}
\label{sec:no-objects}

There is a unique \C-inverse category with $\I_0 = \emptyset$, and a unique diagram in $(\lC^\I)^\Gm$ for every $\Gm$, which is vacuously Reedy fibrant.
(We are using $\Gm$ instead of $X$ because it corresponds to the ambient context in the internal type theory.)
In particular, for any $A\in (\lC^\I)^\Gm$ and $B\in (\lC^\I)^\De$ we have $\lC^\I(A,B) = \Gm\times \De$.
In terms of the internal type theory, this is the unit type regarded as in the context of $\Gm$ and $\De$:
\[ \Gm,\De \vdash \unit \ty. \]

\subsection{One object}
\label{sec:one-object}

Next, suppose that $\I_0 = \{x\}$, and hence the relation $\prec$ is empty.
Then a fibrant \I consists only of a fibrant object $\I(x)$, and an object $A\in (\lC^\I)^\Gm$ is just a map $A_x \to \I(x)\times \Gm$.
Since $x/\I = \emptyset$, we have $M_x A = \I(x)\times \Gm$; thus, $A$ is Reedy fibrant just when $A_x \to \I(x)\times \Gm$ is a fibration.
In terms of the internal type theory of \C, a fibrant \I with $\I_0 = \{x\}$ is just a type in the empty context:
\[ \vdash \I(x) \ty \]
and a (Reedy fibrant) diagram is just a type family
\[ \Gm, (u_x:\I(x)) \vdash A_x(u_x) \ty \]
In this language, the hom $\lC^\I(A,B)$ for $A\in (\lC^\I)^\Gm$ and $B\in (\lC^\I)^\De$ is
\[ \Gm, \De \vdash \prd{u_x:\I(x)} A_x(u_x) \to B_x(u_x) \ty \]
Since $\C^\I \cong \C/\I(x)$, we are just viewing a slice category in a different way.

\subsection{Two objects}
\label{eg:2obj}

Now suppose $\I_0 = \{x,y\}$, with $x\prec y$.
Then a fibrant \C-inverse category \I consists of fibrant objects $\I(x)$ and $\I(y)$ and a Reedy fibrant diagram $\I(y,\blank) \in (\lC^{y/\I})^{\I(y)}$.
By the previous example, this just means a fibration $\I(y,x) \to \I(x)\times \I(y)$; thus in the internal type theory \I consists of
\begin{align*}
  &\vdash \I(x) \ty\\
  &\vdash \I(y) \ty \\
  (u_y:\I(y)), (u_x:\I(x)) &\vdash \I(y,x)(u_y,u_x) \ty
\end{align*}
An object $A\in (\lC^\I)^\Gm$ consists of $A_x \to \I(x)\times \Gm$ and $A_y \to \I(y)\times \Gm$ with a map $A_y \times_{\I(y)} \I(y,x) \to A_x$.
By the previous two examples, its matching object $M_x A$ is
\[ \Gm, (u_x:\I(x)) \vdash \unit \ty \]
and its matching object $M_y A$ is 
\[ \Gm, (u_y:\I(y)) \vdash \tprd{u_x:\I(x)} \I(y,x)(u_y,u_x) \to A_x(u_x) \ty \]
Thus, a Reedy fibrant $A$ consists of
\begin{align*}
  \Gm, (u_x:\I(x)) &\vdash A_x(u_x) \ty\\
  \Gm, (u_y:\I(y)), \left(v_x : \tprd{u_x:\I(x)} \I(y,x)(u_y,u_x) \to A_x(u_x)\right) &\vdash A_y(u_y,v_x) \ty
\end{align*}
In other words, the type $A_y$ is \emph{indexed by} its elements' images in $A_x$ under all the morphisms in $\I(y,x)$.
The hom $\lC^\I(A,B)$ is
\begin{multline}
  \Gm,\De \vdash \tsm{f_x : \prd{u_x:\I(x)} A_x(u_x) \to B_x(u_x)}
  \tprd{u_y:\I(y)}\\
  \tprd{v_x : \tprd{u_x:\I(x)} \I(y,x)(u_y,u_x) \to A_x(u_x)}
  A_y(u_y,v_x) \to B(u_y, f_x \circ v_x) \ty
\end{multline}

\begin{eg}\label{eg:sierpinski}
  Suppose $\I(x) = \I(y) = \unit$ and that $\I(y,x)(u_y,u_x) = \unit$.
  Then up to equivalence, the variables $u_x$ and $u_y$ in the definition of $A$ may be ignored, while $v_x$ reduces simply to an element of $A_x$ (assuming function extensionality).
  Thus, a Reedy fibrant diagram over this \I consists of
  \begin{align*}
    \Gm &\vdash A_x \ty\\
    \Gm, (v_x:A_x) &\vdash A_y(v_x) \ty
  \end{align*}
  which is just the ``Sierpinski topos'' model from~\cite{shulman:invdia}.
\end{eg}

\begin{eg}\label{eg:inverse-cospan}
  Now suppose that $\I(x) = \unit$ while $\I(y) = \bool$, the two-element type with elements $\btrue:\bool$ and $\bfalse:\bool$.
  Let $\I(y,x)(u_y,u_x) = \unit$ for all $u_y$ and $u_x$.
  Then the variable $u_x$ can be disregarded, while a type $A_y$ dependent on $u_y:\I(y)$ consists up to equivalence of two types $A_{y,\btrue}$ and $A_{y,\bfalse}$.
  The type of $v_y$ is again equivalent to $A_x$, so a Reedy fibrant diagram over this $\I$ consists of
  \begin{align*}
    \Gm &\vdash A_x \ty\\
    \Gm, (v_x:A_x) &\vdash A_{y,\btrue}(v_x) \ty\\
    \Gm, (v_x:A_x) &\vdash A_{y,\bfalse}(v_x) \ty.
  \end{align*}
  This is just the ordinary inverse-diagrams model, for the inverse category
  \[ \xymatrix@-1pc{ (y,\btrue)\ar[dr] && (y,\bfalse) \ar[dl] \\ &x. } \]
\end{eg}

\begin{eg}\label{eg:inverse-span}
  Finally, suppose that $\I(x)=\bool$ and $\I(x) = \unit$, with $\I(y,x)(u_y,u_x) = \unit$ for all $u_y$ and $u_x$.
  Then $A_x$ consists up to equivalence of two types $A_{x,\btrue}$ and $A_{x,\bfalse}$, while $u_y$ can be disregarded, and the type of $v_x$ is equivalent to $A_{x,\btrue} \times A_{x,\bfalse}$.
  Thus, a Reedy fibrant diagram over this \I consists of
  \begin{align*}
    \Gm &\vdash A_{x,\btrue} \ty\\
    \Gm &\vdash A_{x,\bfalse} \ty\\
    \Gm, (v_{x,\btrue} : A_{x,\btrue}) , (v_{x,\bfalse} : A_{x,\bfalse}) &\vdash A_{y}(v_{x,\btrue}, v_{x,\bfalse}) \ty
  \end{align*}
  This is again an ordinary inverse-diagrams model, for the inverse category
  
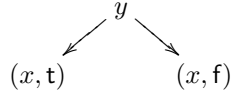
\[ \xymatrix@-1pc{ & y \ar[dl] \ar[dr] \\ (x,\btrue) && (x,\bfalse) } \]
\end{eg}

\subsection{Three objects}
\label{sec:three-objects}

Suppose $\I_0 = \{x,y,z\}$ with $x\prec y \prec z$ (hence $x\prec z$).
Then a fibrant \I consists of
\begin{align*}
  &\vdash \I(x) \ty\\
  &\vdash \I(y) \ty\\
  &\vdash \I(z) \ty\\
  (u_y:\I(y)), (u_x:\I(x)) &\vdash \I(y,x)(u_y,u_x) \ty\\
  (u_z:\I(z)), (u_x:\I(x)) &\vdash \I(z,x)(u_z,u_x) \ty
\end{align*}
and also
\begin{multline*}
  (u_z:\I(z)), (u_y:\I(y)), \left(v_x : \tprd{u_x:\I(x)} \I(y,x)(u_y,u_x) \to \I(z,x)(u_z,u_x)\right) \\
  \vdash \I(z,y) (u_z,u_y,v_x) \ty.
\end{multline*}
Note that a morphism in $\I(z,y)$ is \emph{indexed by} a function $v_x$ assigning its composites with all morphisms in $\I(y,x)$.
In other words, the composition $\I(y,x) \times_{\I(y)} \I(z,y) \to \I(z,x)$ is encoded by type dependency.

A Reedy fibrant diagram $A$ over such an \I consists of
\begin{align*}
  \Gm, (u_x:\I(x)) &\vdash A_x(u_x) \ty\\
  \Gm, (u_y:\I(y)), \left(v_x : \tprd{u_x:\I(x)} \I(y,x)(u_y,u_x) \to A_x(u_x)\right) &\vdash A_y(u_y,v_x) \ty
\end{align*}
and also
\begin{multline*}
  \Gm, (u_z:\I(z)), \left(v_x : \tprd{u_x:\I(x)} \I(z,x)(u_z,u_x) \to A_x(u_x)\right),\\
  \Big(v_y : \tprd{u_y:\I(y)}{w_x:\tprd{u_x:\I(x)} \I(y,x)(u_y,u_x) \to \I(z,x)(u_z,u_x)}\\
  \I(z,y)(u_z,u_y,w_x) \to A_y(u_y,v_x \circ w_x)\Big)
  \vdash A_z(u_z,v_x,v_y) \ty
\end{multline*}

\subsection{Four objects}
\label{sec:four-objects}

Finally, suppose $\I_0=\{x,y,z,w\}$ with $x\prec y\prec z\prec w$.
Then a fibrant \I consists of
\[ \vdash \I(x) \ty\qquad
  \vdash \I(y) \ty\qquad
  \vdash \I(z) \ty\qquad
  \vdash \I(w) \ty\]
\begin{align*}
  (u_y:\I(y)), (u_x:\I(x)) &\vdash \I(y,x)(u_y,u_x) \ty\\
  (u_z:\I(z)), (u_x:\I(x)) &\vdash \I(z,x)(u_z,u_x) \ty\\
  (u_w:\I(w)), (u_x:\I(x)) &\vdash \I(w,x)(u_w,u_x) \ty
\end{align*}
\begin{multline*}
  (u_z:\I(z)), (u_y:\I(y)), \left(v_x : \tprd{u_x:\I(x)} \I(y,x)(u_y,u_x) \to \I(z,x)(u_z,u_x)\right) \\
  \vdash \I(z,y) (u_z,u_y,v_x) \ty.
\end{multline*}
\begin{multline*}
  (u_w:\I(w)), (u_y:\I(y)), \left(v_x : \tprd{u_x:\I(x)} \I(y,x)(u_y,u_x) \to \I(w,x)(u_w,u_x)\right) \\
  \vdash \I(w,y) (u_w,u_y,v_x) \ty.
\end{multline*}
\begin{multline*}
  (u_w:\I(w)), (u_z:\I(z)), \left(v_x : \tprd{u_x:\I(x)} \I(z,x)(u_z,u_x) \to \I(w,x)(u_w,u_x)\right),\\
  \Big(v_y : \tprd{u_y:\I(y)}{w_x:\tprd{u_x:\I(x)} \I(y,x)(u_y,u_x) \to \I(z,x)(u_z,u_x)}
  \I(z,y)(u_z,u_y,w_x) \\ \to \I(w,y)(u_w,u_y,v_x \circ w_x)\Big)
  \vdash \I(w,z)(u_w,u_z,v_x,v_y) \ty
\end{multline*}
Unsurprisingly, a morphism in $\I(w,z)$ is indexed both by a function $v_x$ assigning its composites with all morphisms in $\I(z,x)$, and a function $v_y$ assigning its composites with all morphisms in $\I(z,y)$.
However, since morphisms in $\I(z,y)$ and $\I(w,y)$ are indexed by their composites with morphisms in $\I(y,x)$, the output type of $v_y$ depends on $v_x$.
In this way, type dependency also encodes the \emph{associativity} of composition.
Note that this associativity is judgmental, corresponding to the categorical assumption that the associativity diagrams commute on the nose in \C (rather than up to homotopy).

\subsection{Equivariant homotopy theory}
\label{sec:equiv-homot-theory}

We end with the motivating class of examples.
Let $G$ be a topological group, and let $\OG$ be its \textbf{orbit category}, whose objects are $G$-spaces of the form $G/H$ for closed subgroups $H\le G$, and whose morphisms are $G$-maps.
There is a map $G/H \to G/K$ in $\OG$ if and only if $H$ is conjugate to a subgroup of $K$.
We regard $\OG$ as topologically enriched, so it presents a small \io-category which we also denote $\OG$.
This is an EI \io-category, as is its opposite; but for a general $G$, neither is inverse EI.
For instance, $G=\lR$ has both infinite ascending and descending chains of subgroups
\[\textstyle \lZ < \frac12\lZ < \frac14\lZ < \frac18\lZ < \cdots\qquad
 \lZ > 2\lZ > 4\lZ > 8\lZ > \cdots. \]

But if $G$ is a (finite-dimensional) \emph{compact Lie} group, then $\OG\op$ is an inverse EI \io-category
(\cite[Examples 1.8(e)]{bm:extn-reedy}),
i.e.\ such a $G$ does not have any infinite descending chain of subgroups; this can be proven by assigning to each subgroup $H$ the ordinal $\omega\cdot \dim(H) + |\pi_0(H)|$.
By~\cite{elmendorf:theorem}, the equivariant homotopy theory of $G$-spaces is equivalent to the pointwise homotopy theory of topological diagrams on $\OG\op$, i.e.\ the \io-category $\ig^{\OG\op}$; thus it models homotopy type theory.

\begin{eg}\label{eg:Cp}
  Let $G=C_p$ be the finite cyclic group with $p$ elements, for $p$ a prime.
  Then $G$ has exactly two subgroups, itself and the trivial one $e$, and in \OG we have
  \begin{alignat*}{2}
    \OG(G/e,G/e) &\cong G &\qquad
    \OG(G/G,G/e) &\cong \emptyset\\
    \OG(G/e,G/G) &\cong 1 &\qquad
    \OG(G/e,G/e) &\cong 1.
  \end{alignat*}
  Thus, $\OG\op$ as a fibrant \sSet-inverse category is an instance of \cref{eg:2obj}, with
  \begin{align*}
    \OG\op(G/e) &= BG \qquad\text{(The classifying space of $G$)}\\
    \OG\op(G/G) &= 1\\
    \OG\op(G/G,G/e) &= BG \qquad\text{(More precisely, the span $BG \ot BG \to 1$)}
  \end{align*}
  A Reedy fibrant $\OG\op$-diagram then consists of a fibration $A_{G/e}\to BG$ together with a fibration $A_{G/G} \to \Pi_{BG} A_{G/e}$, or in the type theory
  \begin{align*}
    \Gm, (u_x : BG) &\vdash A_{G/e}(u_x)\\
    \Gm, \left(v_x : \tprd{u_x:BG} A_{G/e}(u_x)\right) &\vdash A_{G/G}(v_x)
  \end{align*}
  If we regard $A_{G/e}$ as a space with a coherent $G$-action, then $\Pi_{BG} A_{G/e}$ is its space of fixed points.
  Thus, $A$ consists of a type with a $G$-action together with, for each fixed point of this action, a type of ``special reasons'' why that point should be considered fixed (which might be empty).
  That is, in passing from the naive homotopy theory of $G$-spaces to $\OG\op$-diagrams, we make ``being a fixed point'' from a property into data.
\end{eg}

\appendix

\section{On the definition of strong fibration functors}
\label{sec:fibfr}

Here I will sketch how to modify the gluing construction from~\cite{shulman:invdia} for the weaker definition of strong fibration functor from \cref{def:fibfr}.
The assumption that $G$ preserves acyclic cofibrations was used in only two places.
The first is to construct explicit factorizations of the diagonal $A \to P_B A \to A\times_B A$ of a Reedy fibration $A\to B$.
Instead, we can use the ordinary Reedy method of factorization, using the path object $(P_B A)_0 = P_{B_0} A_0$ in \C and then letting $(P_B A)_1$ be a factorization of
\begin{equation}
  A_1 \to  (A_1 \times_{B_1} A_1)\times_{(GA_0\times_{GB_0} GA_0)} G(P_{B_0} A_0)\label{eq:Pf-tofactor}
\end{equation}
as an acyclic cofibration followed by a fibration.
It is convenient to construct such a factorization explicitly as follows.
First let $g$ be a lift in the following square:
\begin{equation*}
  \vcenter{\xymatrix{
      GA_0\ar[r]^{G(\r{})}\ar[d]  &
      G(P_{B_0} A_0)\ar[d]\\
      P_{GB_0} (GA_0)\ar[r] \ar@{.>}[ur]^g &
      GA_0 \times_{GB_0} GA_0
      }}
\end{equation*}
By the 2-out-of-3 property and the fact that $G$ preserves homotopy equivalences, $g$ is a homotopy equivalence.
Factor $g$ as an acyclic cofibration followed by a fibration using the mapping path space construction in the slice category of fibrations over $GA_0 \times_{GB_0} GA_0$ (which is isomorphic to $G(A_0 \times_{B_0} A_0)$).
This produces a factorization
\begin{equation}
  \xymatrix{ P_{GB_0} (GA_0) \ar@<-2mm>[r]_-i & Q \ar@<-2mm>[l]_-{q} \ar[r] & G(P_{B_0} A_0)}\label{eq:Pffact1}
\end{equation}
in which the acyclic cofibration $i$ has a retraction $q$ that is a fibration.
Now use the method of~\cite[\S8]{shulman:invdia} but with $P_{G B_0} (G A_0)$ in place of $G(P_{B_0}A_0)$.
This works because $G A_0 \to P_{G B_0}(G A_0)$ is by assumption (unlike $G A_0 \to G(P_{B_0}A_0)$) an acyclic cofibration, and produces an (acyclic cofibration, fibration) factorization
\begin{equation}
  \xymatrix{ A_1 \ar[r]^j &R \ar[r] & (A_1 \times_{B_1} A_1) \times_{(GA_0 \times_{GB_0} GA_0)} P_{G B_0} (G A_0) .}\label{eq:Pffact2}
\end{equation}
Since~\eqref{eq:Pffact1} lies in the slice over $GA_0 \times_{GB_0} GA_0$, it is preserved by pullback along $A_1 \times_{B_1} A_1 \to GA_0 \times_{GB_0} GA_0$.
Combining this pullback factorization with~\eqref{eq:Pffact2} we have the bottom row and right column of the following diagram:
\begin{equation}
  \vcenter{\xymatrix@C=3pc{
    && (A_1 \times_{B_1} A_1)\times_{(GA_0\times_{GB_0} GA_0)} G(P_{B_0} A_0)\\
    T \ar@<-2mm>[d] \ar[r]^{j'} & S \ar[r] \ar@<-2mm>[d] \ar[ur]^p  & (A_1 \times_{B_1} A_1)\times_{(GA_0\times_{GB_0} GA_0)}  Q \ar@<-2mm>[d]_-{q'} \ar[u]\\
    A_1 \ar[r]^j \ar@<-2mm>[u]_-{i''} &
    R \ar[r] \ar@<-2mm>[u] &
    (A_1 \times_{B_1} A_1)\times_{(GA_0 \times_{GB_0} GA_0)} P_{G B_0} (G A_0)
    \ar@<-2mm>[u]_-{i'}.}}
\end{equation}
We define the objects $S$ and $T$ by pullback of $q'$.
Then $i''$ is an acyclic cofibration by \cref{thm:cat8} (since $q'$ is a fibration), and $j'$ is an acyclic cofibration since it is the pullback of $j$ along the fibration $q'$.
Our desired factorization of~\eqref{eq:Pf-tofactor} is then
\[ \xymatrix{ A_1  \ar[r]^{j' i''} & S \ar[r]^p & (A_1 \times_{B_1} A_1)\times_{(GA_0\times_{GB_0} GA_0)} G(P_{B_0} A_0). } \]
In the internal type theory, this means we take $(P_{B} A)_0$ to be
\begin{equation}
(b_0:B_0),\, (a_0: A_0(b_0)) ,\, (a_0' : A_0(b_0)) \vdash \fId(a_0,a_0') \ty\label{eq:rit0}
\end{equation}
and $(P_{B} A)_1$ to be
\begin{multline}\label{eq:rit1}
  (b_0:G B_0),\, (b_1:B_1(b_0)) ,\, (a_0: GA_0(b_0)) ,\, (a_0' : GA_0(b_0)),\, (p_0: G\fId(a_0,a_0'))\\
  (a_1 : A_1(b_0,b_1,a_0)),\, (a_1' : A_1(b_0,b_1,a_0'))\\
  \vdash
  \sm{p_0':\fId(a_0,a_0')} \fId(g(p_0'),p_0) \times \fId((p_0')_*a_1,a_1')
  \ty
\end{multline}
These explicit path-objects are used in~\cite{shulman:invdia} to show that $(\D\dn G)\f$ inherits a ``cloven structure'' from \C and \D, that ``cloven universes'' in \C and \D can be lifted to $(\D\dn G)\f$, and in the analysis of the univalence axiom in $(\D\dn G)\f$.
For the first two, the exact definition does not matter, only that they can be constructed in the internal type theory; while the third can be performed using~\eqref{eq:rit0} and~\eqref{eq:rit1} instead.
Specifically, the types of $p_1$ in~\cite[\autoref{invdia:fig:sectpath}]{shulman:invdia} and $q_1$ in~\cite[\autoref{invdia:fig:secthtpy}]{shulman:invdia} must be replaced by ones derived instead from~\eqref{eq:rit1} above.
But in the next step, we pull back along the map $GV_0 \to GE_0$ that sets $q_0(b_0)$ to $\r {b_0}$ (thereby identifying $a_0$ with $a_0'$).
With~\eqref{eq:rit0} and~\eqref{eq:rit1}, $G(\r{b_0})$ is no longer the same as $\r{G b_0}$, so we get only
\[ \tsm{p_0':\fId(a_0,a_0)} \fId(g(p_0'),G(\r{})) \times \fId((p_0')_*a_1,a_1') \]
However, since $g$ is an equivalence, this type is equivalent to
\[ \tsm{p_0':\fId(a_0,a_0)} \fId(p_0',g^{-1}(G(\r{}))) \times \fId((p_0')_*a_1,a_1'). \]
And since $\tsm{p_0':\fId(a_0,a_0)} \fId(p_0',g^{-1}(G(\r{})))$ is contractible to $\r{}$, this is equivalent to
$\fId(a_1,a_1')$.
Thus, up to equivalence we get the same reduction as in~\cite{shulman:invdia}, so that the univalence axiom still holds.

There is one more use of the assumption that $G$ preserves acyclic cofibrations, in the proof of~\cite[Proposition \ref{invdia:thm:reedy-afib}]{shulman:invdia}, where we use the fact that $\alpha^*(G(P_{A_0} B_0))$ is a path object for $s:B_{01} \to A_1$ (here $B_{01} = A_1 \times_{GA_0} GB_0$).
Under our weaker assumption, we only know that it is a factorization of $B_{01} \to \alpha^*(G(P_{A_0} B_0)) \to B_{01} \times_{A_1} B_{01}$ as an equivalence followed by a fibration.
We may now factor the equivalence $B_{01} \to \alpha^*(G(P_{A_0} B_0))$ as an acyclic cofibration followed by an acyclic fibration:
\[ B_{01} \to P_{A_1} B_{01} \to \alpha^*(G(P_{A_0} B_0)) \]
to obtain an actual path object for $s$.
Since $P_{A_1} B_{01} \to \alpha^*(G(P_{A_0} B_0))$ is an acyclic fibration, it has a deformation section.
Thus, we can lift the ``homotopy'' $hs \sim 1$ using $\alpha^*(G(P_{A_0} B_0))$ to an actual homotopy $H$ using this actual path object $P_{A_1} B_{01}$ satisfying the same equations.
We can now use $P_{A_1} B_{01}$ in place of $\alpha^*(G(P_{A_0} B_0))$ throughout the rest of the proof.

\bibliographystyle{alpha}
\bibliography{all}

\end{document}